\makeatletter
\disable@package@load{program}{}
\makeatother
\documentclass[numbered]{sn-jnl}

\usepackage{graphicx}%
\usepackage{amsmath,amssymb,amsfonts}%
\usepackage{amsthm}%
\usepackage{mathrsfs}%
\usepackage[title]{appendix}%
\usepackage{xcolor}%
\usepackage{textcomp}%
\usepackage{manyfoot}%
\usepackage{booktabs}%
\usepackage{algorithmicx}%
\usepackage{algpseudocode}%
\usepackage{listings}%

\usepackage{hyperref}
\hypersetup{    colorlinks=true,    linkcolor=blue,   filecolor=blue,    urlcolor=blue,    citecolor=blue,    filecolor=blue    }

\newtheorem{theorem}{Theorem}
\newtheorem{lemma}[theorem]{Lemma}
\newtheorem{proposition}[theorem]{Proposition}%
\newtheorem{corollary}[theorem]{Corollary}%

\newtheorem{problem}{Problem}%

\theoremstyle{definition}
\newtheorem{definition}[theorem]{Definition}%
\newtheorem{remark}[theorem]{Remark}

\newtheorem*{remark*}{Remark}

\newcommand\hide[1]{\empty}

\newcommand\todo[1]{ [~ {\color{red} #1 }]}
\renewcommand\todo[1]{\empty}

\newcommand\obsolete[1]{{\color{lightgray}\noindent{\bf Obsolete:} {#1}}}
\renewcommand\obsolete[1]{\empty}

\def\PSpace{\mathrm{PSpace}}

\newcommand\framets[1]{#1}

\def\frF{\framets{F}}

\def\mM{\framets{M}}

\def\tiff{\text{ iff }}

\def\clF{\mathcal{F}}

\def\Log{\myoper{Log}}
\def\Logt{\myoper{Log}\nolimits_\mathtt{t}}
\def\Lt{L_\mathtt{t}}
\newcommand\myoper[1]{\mathop{\myopts{#1}}}
\newcommand\myopts[1]{\mathrm{#1}}

\def\Di{\lozenge}

\def\imp{\rightarrow}

\newcommand\LogicNamets[1]{\logicts{#1}}
\newcommand\logicts[1]{{\textsc{#1}}}
\newcommand\LS[1]{\LogicNamets{S#1}}
\def\vL{L}
\def\Alg{\myoper{Alg}}

\def\clU{\mathcal{U}}
\def\clV{\mathcal{V}}

\def\clC{\mathcal{C}}

\def\clV{\mathcal{V}}

\def\clP{\mathcal{P}}
\def\EE{\exists}
\def\AA{\forall}

\def\restr{{\upharpoonright}}

\def\v{\theta}
\def\vext{\bar{\v}}

\def\vf{\varphi}
\def\mo{\vDash}

\def\emp{\varnothing}

\def\se{\subseteq}
\newcommand\languagets[1]{\logicts{#1}}

\def\PV{\languagets{PV}}
\def\wK4{\logicts{wK4}}

\newcommand{\ff}[1]{\widehat{#1}}
\newcommand{\fm}[1]{\hat{#1}}

\newcommand{\LK}[1]{\logicts{K#1}}
\newcommand\lng[1]{{\ell(#1)}}

\def\Lm{\logicts{L}_m}
\def\Km{\logicts{K}_m}

\def\val{\theta}
\def\valext{\bar{\val}}

\def\at{\mathtt{at}}

\newcommand\ISh[1]{{~\bf IS}: {\color{teal} #1}}
\renewcommand\ISh[1]{\empty}
\newcommand\IShLater[1]{{~\bf IS, later}: {\color{teal} #1}}
\renewcommand\IShLater[1]{\empty}
\newcommand\IS[1]{\ISh{#1}}
\newcommand\AK[1]{{~\bf AK}: {\color{orange}#1}}
\renewcommand\AK[1]{\empty}
\newcommand{\set}[1]{\left\{#1\right\}}

\def\Sub{\myoper{Sub}}

\renewcommand\mod{\,\mathrm{mod}\,}
\newcommand\rem[1]{\langle{#1}\rangle}

\newcommand\lngth[1]{\lng{\vf}}

\begin{document}

\title{Two types of filtrations for $\wK4$ and its relatives}

\affil[1]{\orgname{Higher School of Modern Mathematics MIPT},  \city{Moscow}}
\affil[2]{\orgname{HSE University}, \city{Moscow}}
\affil[3]{\orgname{New Mexico State University}}

\author[1,2]{
\fnm{Andrey}  \sur{Kudinov}
  \email{akudinov@hse.ru}}
\author[3]{
  \fnm{Ilya}   \sur{Shapirovsky}\footnote{The work of this author was supported by NSF Grant DMS - 2231414},
  \email{ilshapir@nmsu.edu}
}

%

%
%



%
%
\abstract{
We study the finite model property of subframe logics with expressible transitive reflexive\IShLater{why?}
closure modality.
For $m>0$, let $\Lm$ be the logic 
defined
by axiom $\Di^{m+1} p\imp \Di p\vee p$.
We construct quotient filtrations for the logics $\Lm$, which implies that these logics and their tense counterparts
have the finite model property. 
Then, we construct selective filtrations of the canonical models of $\Lm$, which implies
that all canonical subframe logics containing $\Lm$ have the finite model property.
\IShLater{Discuss the wording}
}


\maketitle    %
\section{Introduction}
There are two widely used methods for proving the finite model property of a modal logic $L$. 
One method is based on transforming a given model $M$ of $L$ into its finite quotient, known as a {\em filtration of $M$} \cite[Section 5.3]{CZ}. The other method results in a finite submodel of $M$, which is called a {\em selective filtration of $M$} \cite[Section 5.5]{CZ} or {\em selection} \cite[Section 2.3]{BDV}. 
Moreover, in  many cases  these constructions imply the finite model property for 
related logics, 
in particular for counterparts of $L$ in  extended languages
or for   extensions of $L$ in the same language. 

We are interested in modal logics with expressible transitive reflexive closure modality. Such logics are said to be {\em pretransitive}. 
In terms of frames, this means that for some fixed finite $m$, in frames of the logic we have
\begin{equation*} \textstyle
R^{m+1}\subseteq \bigcup_{i\leq m} R^i,
\end{equation*}
\AK{I added "\textbackslash limits" I think it looks better, but I do not insist.}
\IS{OK, but then we need to change the style throughout all the paper. Please do it in the next version, if you like it more, but make sure that the notation is consistent. It is about an hour of work, I think....}
where $R^0$ is the diagonal relation,  $R^{i+1}=R\circ R^i$, and $\circ$ is the composition. We address this property as  {\em $m$-transitivity}.

The simplest non-transitive example of a pretransitive logic is the smallest 1-transitive logic $\wK4$, which is the logic of
{\em weakly transitive relations} given by the condition
$R^2\subseteq R\cup R^0$, or equivalently,  $x R y R z$ implies $xR z$ or $x=z$.
In  \cite{esakia2001weak}, this logic was studied in the context of topological semantics of modal logic; it was shown
that $\wK4$ is the logic of the class of all topological spaces, where $\Di$ is interpreted as the Cantor's derivative operator. 
In spite of its similarity to the logic $\LK{4}$  of transitive relations,
$\wK4$ is more difficult to study.
The finite model property of $\wK4$ is 
shown
in \cite{GuramEsakiaDavid2009}.
In \cite{MaChen2023}, it was shown 
that its expansion with the inverse modality (tense expansion) has the finite model property.

\ISh{Something like this:
...
The case $m>1$ is much more intriguing even for the subframe case. The main results of the paper pertains to the case $m>1$....}
\AK{I agree but I'd better replace "pertains" with "dedicated".}
\IS{That was a draft. Let us discuss it later. }
For $m>1$, the finite model property of the logic of all $m$-transitive relations is an open problem   \cite[Problem 11.2]{CZ}.  \IShLater{formally, different problem}
In \cite{PartitionaingKripkrFrames}, filtrations were constructed for $m$-transitive relations with bounded height of their skeletons.
In \cite{Gabbay:1972:JPL}, filtrations were constructed for  relations satisfying
\begin{equation}\label{eq:prop:gabb}
R^{m+1}\subseteq R;
\end{equation}
these logics can be thought as $m$-transitive generalizations of $\LK{4}$.
As well as weak transitivity,  and unlike $m$-transitivity for $m>1$,
the property \eqref{eq:prop:gabb} is preserved
under taking substructures.  Logics of such classes are called {\em subframe}.
\IShLater{Inaccurate. This is a theorem, see Wolter, page 25, final paragraph.}
\AK{(?)Logics whose classes are closed under taking subframes are called \emph{subframe logics}.}
\IS{Still inaccurate: there are incomplete subframe logics, I believe. Let us double check it later.}

\medskip
\hide{
In this paper, 
We consider $m$-transitive analogs $\Lm$ of the logic $\wK4$:  $\Lm$ is the logic of relations such that $$R^{m+1}\subseteq R\cup R^0.$$
}

For $m>0$, we consider the following $m$-transitive analog 
of weak transitivity: 
$$R^{m+1}\subseteq R\cup R^0;$$ 
$\Lm$ denotes the logic of such relations.  
We give a proof of the finite model property for all logics $\Lm$, 
for their tense expansions, and for  the canonical subframe extensions of $\Lm$.

\smallskip
In Theorem \ref{thm:ADF-Lm}, we describe filtrations for $\Lm$. Hence, these logics have the finite model property and are decidable.
Moreover, their tense expansions have the finite model property:
according to \cite{KSZ:AiML:2014}, if a class of frames admits filtration,  then the tense expansion of its logic has the finite model property.
Our construction is essentially based on the fact that the logic of the class of clusters occurring in $\Lm$-frames is locally tabular  \cite{LocalTab16AiML}.
\IShLater{Tense expansion? BAAD}

Then, we use another approach to study subframe extensions of $\Lm$.
The case $m=1$ is well studied: 
it is known that all subframe extensions of $\LK{4}$ have the finite model property \cite{Fine85}; this result was
generalized for subframe extensions of $\wK4$ in \cite{wk4Subfrae}.
For $m>1$, it is unknown whether all $m$-transitive  subframe logics have the finite model property; \IShLater{Check with Frank}
an axiomatic characterization of $m$-transitive subframe logics (given for $m=1$ in \cite{Fine85} and   \cite{wk4Subfrae}) is also unknown
for  $m>1$.    
We 
construct selective filtrations of the canonical models of $\Lm$ and 
prove that every canonical subframe extension of $\Lm$ has the finite model property. Our construction is essentially based on the maximality property of pretransitive canonical frames:
it was shown in \cite{Fine85} that in canonical frame of $\LK{4}$, 
every non-empty definable subset has a maximal element \cite{Fine85};
this property generalizes for the pretransitive case   \cite{gliv}. 

\IShLater{DoubleCheck: Gliv is first}
\IShLater{DoubleCheck: no pretransitive counterexampels}

\hide{
Modal logic $\wK4 = K + \Di\Di p \to \Di p \lor p$ is a weaker version of the well known $K4 = K + \Di\Di p \to \Di p$ logic. Both of these logics are closely related to the derivational semantics (d-semantics) which is based on topological spaces where modality $\Di$ is interpreted as the derived set operator. This interpretation was introduced in \cite{McKinseyTarski44} by McKinsey and Tarski. And even earlier in \cite{Kuratowski1922} Kuratowski suggested formulas with the derived operator that should be valid in topological spaces. The Kuratowski paper was purely topological without logic. Unfortunately, Kuratowski and later McKinsey and Tarski were interested in well behaving spaces so they only noticed that $K4$ is valid in $T_1$ spaces and did not look for completeness with respect to the class of all topological spaces. It was Esakia \cite{esakia2001weak} who defined logic $wK4$ and proved completeness w.r.t. all topological spaces. Unfortunately, the proof of finite model property (fmp) for $wK4$ in \cite{esakia2001weak} had a mistake, which was corrected in \cite{GuramEsakiaDavid2009}.

D-semantics is widely used in the studies of topological modal logics (see \cite{}) and in provability logic (see \cite{}) \AK{I am not sure if it is needed, I can look for cites if it does.}.
By adding the symmetry axiom to $wK4$ we get logic $DL$ which is the logic of inequality (or the difference logic). The difference modality is also very useful and is used in the topological semantics to increase expressivity (see \cite{Kudinov2014}).

One way to generalize a transitive modal logic is to consider logics of the form $K + \Di^{m+1} p \to \Di p$ and it was done in \cite{Gabbay:1972:JPL}. Gabbay proved that such logics have fmp and decidable. We will call such logics pretransitive.

In this paper we consider weak pretransitive logics $L_m = K + \Di^{m+1} p \to \Di p \lor p$ and proof that they have fmp and decidable.

The temporal version of $wK4$ was studied in \cite{MaChen2023}. It was proved that the temporal $wK4$ has fmp and decidable. In this paper we  generalize this result and using \cite{Ai} prove that temporal $L_m$ has fmp and decidable..
}

\section{Preliminaries}

We are following standard definitions of the theory of modal logics, see, e.g., \cite{CZ} or \cite{BDV}.
\subsection{Basic notions}

\paragraph*{Language}

The set of {\em modal formulas} is built from
a countable set of {\em variables} $\PV=\{p_0,p_1,\ldots\}$ using Boolean connectives $\bot,\imp$ and a unary connective $\Di$.
Other logical connectives
are defined as abbreviations in the standard way, in particular $\Box\vf$
denotes $\neg \Di \neg \vf$.

\paragraph*{Frames and algebras}

A {\em  (Kripke) frame} is a pair  $\frF=(X,R)$,
where $X$ is a set and $R\se X{\times}X$.
For $a\in X$, $Y\subseteq X$, we put $R(a)=\{b\mid aR b\}$,
$R[Y]=\bigcup_{a\in Y} R(a)$.

A {\em valuation in a frame $F$} is a map $\PV \to \clP(X)$,
where $\clP(X)$ is the set of subsets of $X$.
A {\em (Kripke) model on} $\frF$ is a pair $(\frF,\theta)$, where $\theta$ is a valuation.
The \emph{truth} and \emph{validity} are defined in the usual way, in particular
${(\frF,\theta),a\mo\Di \vf}$ means that  ${(\frF,\theta),b\mo\vf}$  for some $b$ in $R(a)$.
We put
$$\vext(\vf)=\{a\mid (F,\v),a\mo\vf\}.$$
A formula $\vf$ is {\em true in a model $\mM$}, if $\mM,a\mo\vf$ for all $a$ in $\mM$;
it is  {\em valid in a frame $\frF$}, in symbols $\frF\mo\vf$,
if $\vf$ is true in every model on $\frF$.
A formula is {\em valid in a class} $\clF$ of frames, if it is  valid in every frame $\frF\in\clF$.
Validity of a set of formulas means validity of every formula in this set.

A {\em modal algebra} is a Boolean algebra endowed with a unary operation $\Di$  that distributes with respect to finite joins: $\Di (a\vee b)=\Di a\vee \Di b$,  $\Di \bot =\bot$. 
The {\em (modal) algebra $\Alg(F)$ of a frame}  $\frF=(X,R)$ is the
powerset Boolean algebra of $X$ endowed with the unary operation $\Di_R$: for $Y\subseteq X,$
$\Di_R(Y)=R^{-1}[Y].$

\paragraph*{Logics}
A ({\em propositional normal modal}) {\em logic} is a set $\vL$ of formulas
that contains all classical tautologies, the axioms $\neg\Di \bot$  and
$\Di  (p_0\vee p_1) \imp \Di  p_0\vee \Di  p_1$, and is closed under the rules of modus ponens,
substitution and {\em monotonicity}: $\vf\imp\psi\in \vL$ implies $\Di \vf\imp\Di \psi\in \vL$.
For a modal formula $\vf$,  $\vL+\vf$ is the smallest logic that contains $\vL\cup\{\vf\}$. The smallest logic is denoted by $\LK{}$.

An {\em $L$-frame} is a frame where $L$ is valid.

The set $\Log{\clF}$ of all formulas that are valid in a class $\clF$ of frames is called the {\em logic of} $\clF$. It is straightforward that $\Log{\clF}$ is a  normal modal logic.
A logic $L$ is {\em Kripke complete}, if $L$ is the logic of a class of frames.
A logic has the {\em finite model property},  if it is the logic of a class of finite frames.

A Kripke complete logic is said to be {\em subframe}, if for every $L$-frame $F$ and every non-empty subset $Y$ of $F$, the restriction $F\restr Y=(Y, R \cap (Y\times Y))$ is also an $L$-frame
(in this paper, we do not consider incomplete subframes logics; the general definition of a subframe logic is given in \cite[Section 2.2]{WolterLat1993}).
\IShLater{Check: non-empty subset $Y$}

\paragraph*{Pretransitivity, skeleton, and clusters}

For a binary relation $R$ on a set $X$,
let $R^*$ denote its transitive reflexive closure $\bigcup_{i <\omega} R^i$, where
where $R^0$ is the diagonal $Id_X=\{(a,a)\mid a\in X\}$ on $X$, $R^{i+1}=R\circ R^i$, and $\circ$ is the composition.

A relation is said to be {\em $m$-transitive}, if
$R^{m+1}\subseteq \bigcup_{i\leq m} R^i$; from this, it follows that $R^*=\bigcup_{i\leq m} R^i$.
The $m$-transitivity  is expressed by the modal formula
$\Di^{m+1} p \imp \bigvee_{i\leq m} \Di^i p$.  A logic $\vL$ is said to be {\em $m$-transitive}, if it contains this formula.
A logic is {\em pretransitive}, if it is $m$-transitive for some $m\geq 0$.

A {\em cluster} in a frame $\frF=(X,R)$ is an equivalence class modulo $\sim_R\; =\{(a,b)\mid a R^* b \text{ and } b R^* a\}$. 
For clusters $C, D$, put $C\leq_R D$, if $a R^* b$ for some $a\in C, b\in D$.
The poset $(X{/}{\sim_R},\leq_R)$ is called the {\em skeleton of} $\frF$.
For $a\in X$, $[a]_R$ denotes the cluster of $a$.

A frame $(C,R)$ is called a {\em cluster frame}, if $R^*\,=\,C\times C$.
Clearly, if $C$ is a cluster in a frame $F=(X,R)$, then the restriction 
$F\restr C$
is a cluster frame.

\subsection{Filtrations}\label{sec:Filtr}

Let $\sim$ be an equivalence on a set $X$. The equivalence class of $x\in X$ modulo $\sim$ is denoted by $[x]_\sim$.
\hide{
The {\em index of $\sim$ } is the cardinality of the quotient $X/{\sim}$.}
For  equivalences $\sim_1,\sim_2$ on $X$, we say that {\em $\sim_1$ refines $\sim_2$}, if $\sim_1\;\subseteq \;\sim_2$.

For a set $\Gamma$  of formulas,
let $\Sub(\Gamma)$ denote the set of all subformulas of formulas from~$\Gamma$.
We say that $\Gamma$ is \emph{$\Sub$-closed}, if ${\Sub(\Gamma)\se\Gamma}$.

For a model $M = (X, R, \val)$ and a set  $\Gamma$ of formulas, put
$$ 
\text{$a\sim_{\Gamma} b$ \quad iff \quad$\forall \psi\in\Gamma\; (M,a\models \psi \tiff M,b\models \psi)$.}
$$
The equivalence $\sim_\Gamma$ is said to be {\em induced by $\Gamma$ in $M$}.

\smallskip

\begin{definition}\label{def:epi}
Let $\Gamma$ be a $\Sub$-closed set of formulas,
\IShLater{Do we need it? Yes, even in the definition: otherwise maximal is not well-defined}
$M=(X,R,\val)$ a model.

A {\em filtration of $M$ through $\Gamma$}
is a model $\ff{M}=(\ff{X},\ff{R},\ff{\theta})$
such that
\begin{enumerate}
\item $\ff{X}=X/{\sim}$ for some equivalence relation $\sim$, which refines $\sim_\Gamma$.
\item ${\ff{M},[a]\models p}$ iff ${M,a\models p}$, for all variables $p\in \Gamma$;
here $[x]$ is the $\sim$-class of $x$.
\item $R_{\sim} \subseteq \ff{R} \subseteq  R_{\sim,\Gamma}$, where
$$
\begin{array}{ccl}
~[a]\,{R}_\sim\,[b] & \text{iff} & \exists a'\sim a\ \exists b'\sim
b\;
(a'\,R\,b'),
\smallskip \\
~[a]\,{R}_{\sim,\Gamma}\,[b] & \text{iff} & \forall \vf\;
   (\Di \vf\in \Gamma \: \& \: M,b\models\vf \Rightarrow M,a\models \Di \vf ).
\end{array}
$$
\end{enumerate}
The relations $R_\sim$ and $R_{\sim,\Gamma}$ on $\ff{X}$ are called
the \emph{minimal} and the \emph{maximal filtered relations}, respectively.
\end{definition}

The following fact is standard:
\begin{lemma}[Filtration lemma]\label{Lemma:Filtration}
Suppose that $\Gamma$ is a
$\Sub$-closed set of formulas
and $\ff{M}$ is a $\Gamma$-filtration of a model~$M$. 
Then, for all points $a$ in $M$ and all formulas ${\varphi \in\Gamma}$,
we have: \IShLater{Why finite? Removed; DC}
\begin{center}
${M,a\models \varphi}$ iff ${\ff{M},[a]\models\varphi}$.
\end{center}
\end{lemma}
\begin{proof}
Straightforward induction on $\vf$.
\end{proof}

\begin{proposition}\label{prop:min-included-in-max}
 For a $\Sub$-closed set $\Gamma$ of formulas and a model $M=(X,R,\val)$,
 for any equivalence $\sim$ which refines $\sim_\Gamma$,
we have $R_\sim \, \subseteq \,R_{\sim,\Gamma}$.
\end{proposition}
\begin{proof}
Let $U_1, U_2\in X/{\sim}$, $a_1\, \in U_1$, $a_2 \in U_2$, and $a_1 R a_2$.
Assume $\Di\vf\in\Gamma$. Then $\vf\in \Gamma$.
Assume $M,b\mo\vf$ for some $b\in U_2$. Since $\sim$ refines $\sim_\Gamma$, we have $M,a_2\mo \vf$. Then  $M,a_1\mo\Di\vf$, and so
 $M,a\mo \Di\vf$ for all $a\in U_1$.
Hence $U_1 R_{\sim,\Gamma} U_2$.
\end{proof}

Let $\clF$ be a class of frames. By an {\em $\clF$-model} we mean $M=(F,\theta)$ with $F\in \clF$.

\smallskip

\begin{definition}
A class $\clF$  of frames \emph{admits filtration}
if, for every finite $\Sub$-closed set of formulas~$\Gamma$
and every $\clF$-model $M$,
there exists a finite filtration $\ff{M}$ of $M$ through $\Gamma$ based on a frame in $\clF$.
\end{definition}

\smallskip

Clearly, if $\clF$ admits filtration, then its logic has the finite model property.

For many logics, filtrations based on the initial equivalence $\sim_\Gamma$
work well: this is the most standard application of the method. \IShLater{Refs } For example, for the case  of transitive frames,
 the transitive closure of the minimal filtered relation on the quotient $X/{\sim_\Gamma}$ will give a filtration.

In general, however,  letting the equivalence $\sim$ be finer than $\sim_\Gamma$ gives much more flexibility for the method of filtration.
As an important illustration, consider the logics
 $\Km=\LK{}+\Di^{m+1} p \to \Di p$.
Filtrations for these logics were constructed in \cite{Gabbay:1972:JPL}.
 The class of $\Km$-frames is characterized by the condition
 \begin{equation}\label{eq:Gabbay}
R^{m+1}\subseteq R.
 \end{equation}
For a relation $S$, define its {\em $m$-closure} as $\bigcup_{i<\omega} S^{mi+1}.$
It can be shown that the $m$-closure of a relation $S$ is the least relation $R$ that contains $S$ and satisfies \eqref{eq:Gabbay}.
For a model on a $\Km$-frame and $\Gamma$, to build a filtration, consider the equivalence
induced by $\Delta=\Gamma\cup \{\Di^i\vf \mid i\leq m\, \&\, \Di\vf \in \Gamma\}$; the filtered relation is defined as
the $m$-closure of the minimal filtered relation, see \cite{Gabbay:1972:JPL} for the details.

\IShLater{Segerberg and $\LS{4.1}$}

In \cite{KSZ:AiML:2014,KikotShapZolAiml2020}, filtrations in the sense of Definition \ref{def:epi} were used to transfer
the finite model property for logics in extended languages, in particular for tense expansions.

\subsection{Tense expansions}
\IShLater{more rigorous way; more refs}

Consider the propositional language with two modalities, $\Di$ and $\Di^{-1}$.

For a frame $(X,R)$, the modality $\Di^{-1}$ is interpreted by the inverse $R^{-1}=\{(a,b)\mid (b,a)\in R\}$ of
$R$.  The set of bimodal formulas that are valid in a class $\clF$ of frames is denoted $\Logt(\clF)$. 

\begin{theorem}\cite{KSZ:AiML:2014}\label{thm:transfer-temporal}
If a class $\clF$ of frames admits filtration, then $\Logt(\clF)$ has the finite model property.
\end{theorem}

The {\em tense expansion $\Lt$ of a logic $L$} 
is defined as
the smallest normal  logic in the language of $\Di$ and $\Di^{-1}$ that contains $L$ and 
two extra axioms 
\begin{equation}\label{eq:tenseAxioms}
p\imp \Box\Di^{-1} p, \quad p\imp \Box^{-1}\Di p, 
\end{equation}
where $\Box^{-1}$ abbreviates $\neg \Di^{-1}  \neg$.
It is well known that if $\Lt$ is Kripke complete, then 
$$\Lt=\Logt(\clF),$$ where $\clF$ is the class of frames of $\vL$. In particular, this is the case for a canonical $L$, since 
\eqref{eq:tenseAxioms} are canonical formulas.

It follows from Theorem \ref{thm:transfer-temporal} that if the class of frames of a finitely axiomatizable logic $L$ admits filtration, and $\Lt$ is Kripke complete,  then $\Lt$ 
is decidable.




\section{Filtrations for the logics \texorpdfstring{$\Lm$}{Lm}}\label{sec:wk4}

\IS{May-June 2025: checked everything}

For $m>0$, 
let $\Lm$ be the logic $\LK{}+\Di^{m+1} p \imp \Di p\vee p$.  
For $m=1$, $\Lm$ is the logic $\wK4$.

It is straightforward that $(X,R)$ validates $\Lm$
iff
\begin{equation}\label{eq:m-property}
R^{m+1} \subseteq \;R\cup Id_X .
\end{equation}
The formulas $\Di^{m+1} p \imp \Di p\vee p$ are Sahlqvist \cite[Section 10.3]{CZ}. Hence, we have: \IShLater{refs; defs}

\smallskip 

\begin{proposition}
All $\Lm$ are canonical and hence are Kripke complete. 
\end{proposition}

\smallskip 

We will show that for each $m$, the class of $\Lm$-frames admits filtration. The general case requires a number of technical steps.
The proof for the  case $m=1$ is much simpler, but illustrates the idea well, so first we provide 
a separate proof for logic $\wK4$.

\subsection{Filtrations for \texorpdfstring{$\wK4$}{wK4}}

Observe that in any $\wK4$-frame $(X,R)$, we have: 
\begin{equation}\label{eq:basicWk4}
\text{if $a R^*b$ and $[a]_R\neq [b]_R$, then $a R b$.}
\end{equation}

For a relation $R$, let $R^+$ be the transitive closure of $R$:  $R^+=\bigcup_{0<i<\omega} R^i$.

\smallskip

\begin{theorem}\label{thm:wK4AF}
The class of $\wK4$-frames admits filtration.
\end{theorem}
\begin{proof}\ISh{Carefully read on May 28, 2025}
Let $M=(F,\theta)$ be a model on a $\wK4$-frame $F=(X,R)$, $\Gamma$ a finite set of formulas closed under taking subformulas.

For a point $a$ in $M$, let
$\Gamma(a)=\{\vf\in \Gamma \mid M,a\mo\vf \}.$
For a cluster $C$ in $F$, consider the set $S(C)=\{\Gamma(a)\mid a\in C\}$.
For $a,b\in X$, put
$$a\approx b \textrm{ iff } S([a]_R)=S([b]_R).$$
Let
$\equiv$ be the intersection $\sim_\Gamma\cap \approx$, that is
$$a\equiv b \tiff  (\Gamma(a)=\Gamma(b) \textrm{ and } S([a]_R)=S([b]_R)).$$
Put $\ff{X}=X{/}{\equiv}$. 
Let $[a]$ denote the $\equiv$-class of $a\in X$.
Observe that $\ff{X}$ is finite. \IShLater{ details?} 

Consider the minimal filtered relation $R_\equiv$ on $\ff{X}$.
Put $$\ff{R}=((R_\equiv)^+\setminus Id_{\ff{X}}) \cup (R_\equiv\cap Id_{\ff{X}}).$$
Since the reflexive closure of $\ff{R}$ is transitive, $(\ff{X},\ff{R})\text{ validates } \wK4.$ \IShLater{Reflexive closure - undefined}
Clearly, $\ff{R}$ contains $R_\equiv$.  

It remains to check that $\ff{R}$ is contained in the maximal filtered relation $R_{\equiv,\Gamma}$.

Since $(R_\equiv\cap Id_{\ff{X}})\,\subseteq R_\equiv$ and $R_\equiv\,\subseteq \, R_{\equiv,\Gamma}$, we have
\begin{equation}\label{eq:fr:max-a}
  R_\equiv\cap Id_{\ff{X}}\, \subseteq \,R_{\equiv,\Gamma}.
\end{equation}
 
Now by induction on $n$ we show that for all $n>0$,
\begin{equation}\label{eq:fr:max-b}
(R_\equiv)^n\setminus Id_{\ff{X}}\, \subseteq \, R_{\equiv,\Gamma}.
\end{equation}
The basis $n=1$ is immediate, since
$(R_\equiv\setminus Id_{\ff{X}})\,\subseteq R_\equiv\,\subseteq \, R_{\equiv,\Gamma}$. \IShLater{Start with $n=0$.}

Let $n>1$.
Suppose that $[a](R_\equiv)^n [b]$, $[a]\neq [b]$, $\Di\vf \in \Gamma$, and $M,b\mo \vf$.
We need to show that $M,a\mo\Di\vf$.

First, assume that $a\approx b$.
We have $\vf \in \Gamma(b)\in S([b]_R)=S([a]_R)$, so $\Gamma(c)=\Gamma(b)$ for some $c \in [a]_R$.
Since $[a]\neq [b]$, $\Gamma(a)\neq \Gamma(b)$. So $\Gamma(c)\neq \Gamma(a)$, and hence $a\neq c$.
Thus $aRc$.
We have $M,c\mo \vf$, so $M,a\mo \Di \vf$.

Now assume that $a\not\approx b$.
Since $[a] (R_\equiv)^n [b]$,
for some $u, v$
we have:  $a\approx u$, $a\not \approx v$, 
$[a] (R_\equiv)^l [u]$,  $uRv$, and $[v] (R_\equiv)^k  [b]$, 
where 
$l+k+1 = n$ (in the $R_\equiv$-path from $[a]$ to $[b]$, $[v]$ is the first $\equiv$-class
such that $a\not \approx v$).  
Then $v\not \approx u$, and so $[v]_R\neq [u]_R$.   

If $[v] = [b]$, then $M,v\mo \vf$.
If $[v] \neq [b]$, then $([v],[b])\in (R_\equiv)^{k}\setminus Id_{\ff{X}}$, and by induction
hypotheses, $[v] R_{\equiv,\Gamma} [b]$; in this case, $M,v \mo \Di \vf$.
In either case, $M,v \mo \vf \vee \Di \vf$.
Since $u$ and $v$ belong to different clusters of a $\wK4$-frame $F$,  and $u  R v $, 
then for every $c\in [u ]_R$ we have 
$M,c\mo \Di \vf$ in view  of \eqref{eq:basicWk4}.
It follows that for every $\Delta \in S([u]_R)$, $\Di \vf \in \Delta$.
We have $\Gamma(a)\in S([a]_R)=S([u]_R)$ (since  $a\approx u$), so $\Di\vf \in \Gamma(a)$. Thus, $M,a\mo \Di \vf$.
\hide{
\IShLater{The logic of the below can be simplified.... perhaps}
Since $a (R_\equiv)^n b$,
for some $u_0,\ldots,u_{n}, v_1,\ldots, v_{n}$
we have:  $a\equiv u_0$,
$u_i R v_{i+1}\equiv u_{i+1}$ for $i< n$,
$b=u_n$.
Consider the set $I=\{i\leq n\mid u_i\approx a\}$.
This set is non-empty, since $u_0\in I$.
Let $k$ be the largest in $I$.
Then $v_{k+1}$ is not in the cluster of $u_k$:
indeed, we have $a\approx [u_k]$ and $v_{k+1}\equiv u_{k+1}$, so $[u_k]_R=[v_{k+1}]_R$ would imply $a\approx u_{k+1}$.

If $[v_{k+1}]=[b]$, then $M,v_{k+1}\mo \vf$.
If $[v_{k+1}]\neq [b]$, then $([v_{k+1}],[b])\in (R_\equiv)^{n-k-1}\setminus Id_{\ff{X}}$, and by induction
hypotheses, $[v_{k+1}] R_{\equiv,\Gamma} [b]$; in this case, $M,v_{k+1} \mo \Di \vf$.
So in both cases, $M,v_{k+1}\mo \vf \vee \Di \vf$.
Since $[u_k]_R\neq [v_{k+1}]_R$ and $u_k R v_{k+1}$, for every $c\in [u_k]_R$ we have $M,c\mo \Di \vf$.
It follows that for every $\Delta \in S(u_k)$, $\Di \vf \in \Delta$.
We have $\Gamma(a)\in S(a)=S(u_k)$, so $\Di\vf \in \Gamma(a)$. This means that $M,a\mo \Di \vf$.
}
\end{proof}

\hide{
\begin{remark*}
Filtrations of $\wK4$-frames were considered in \cite{esakia2001weak}. However, the quotient was given by the equivalence induced by the initial set of formulas $\Gamma$, and the resulting model may not be a filtration. The finite model property of $\wK4$ was proved in \cite{GuramEsakiaDavid2009}. 
\end{remark*}
}

\subsection{Properties of $\Lm$-frames} 
\IShLater{Did we say that we assume that $m>0$?} \AK{Yes!}

 
The following proposition establishes that for $m>0$, the logics $\Lm$ are pretransitive and subframe.

\smallskip

\begin{proposition}\label{eq:Lm-trans-and-subfr} Let $m>0$, 
$F=(X,R)$ be an $\Lm$-frame.  
Then:
\begin{enumerate}
  \item $F$ is $m$-transitive. 
  \item For every  $Y\subseteq X$, $F\restr Y$ is an $\Lm$-frame. 
\end{enumerate}
\end{proposition}
\begin{proof}
Immediate from \eqref{eq:m-property}.  
\end{proof}

The following fact generalizes \eqref{eq:basicWk4} 
for $\Lm$-frames:
\begin{proposition}\label{prop:Lm-diffClusters}
Let $(X,R)$ be an $\Lm$-frame. For $a,b\in X$, we have  
\begin{equation}\label{eq:generalPropLM}
\text{if $aR^{mn+1} b$ for some $n$ and $[a]_R\neq [b]_R$, then $a R b$.}
\end{equation}
\end{proposition}
\begin{proof}
Induction on $n$. Basis $n=0$ is trivial. 
Assume $aR^{mn+m+1} b$. Then for some $c,c'$, we have $aR^m c R c'R^{mn} b$.

If $[a]_R\neq [c']_R$, then $aRc'$ by \eqref{eq:m-property}, and so $aR^{mn+1}b$. By induction hypothesis, $aRb$.

If $[a]_R=[c']_R$, then  $[a]_R=[c]_R$. Hence $[c]_R\neq [b]_R$. By induction hypothesis, $cRb$.
Now $aR^{m+1}b$, and so $aRb$ by \eqref{eq:m-property}.
\end{proof}

 \smallskip

\begin{definition}
A cluster frame $G=(X,R)$  
is said to be {\em $m$-regular}, if $R^{m+1}\subseteq R$; otherwise
$G$ is said to be {\em $m$-irregular}.
A cluster $C$ in a frame $F$ is {\em $m$-(ir)regular}, if the frame F$\restr C$ is.
\end{definition}

\smallskip 


\begin{proposition}\cite[Lemma 6.5]{LocalTab16AiML}\label{prop:irregAiML}
Let  $F=(X,R)$ be a cluster frame, $F\mo \Lm$.
If $F$ is $m$-irregular, then:
\begin{enumerate}
    \item for some $a_0, a_1, \ldots,  a_{m}$ we have 
$a_0R a_1 R\ldots R a_{m} R a_0$;
\item 
$X=\{a_0,\ldots,a_{m}\}$ or $R\cup Id_X=X\times X$,
\end{enumerate}
\end{proposition}

\smallskip 

\begin{remark}
It is not assumed that $a_0,\ldots,a_m$ are $m+1$ distinct points.
\end{remark}
Hence, an $m$-irregular $\Lm$-cluster is an $\wK4$-frame, or a small frame with 
a cyclic path through all its points.

\smallskip 

\begin{proposition}\label{prop:LmIrregular}
Let  $F=(X,R)$ be an $\Lm$-frame, $C$ an $m$-irregular cluster in $F$.
Assume that 
$a\in C$, $b\notin C$. Then:
\begin{enumerate}
\item \label{prop:LmIrregular:1}
if $aR^*b$, then $a R^i b$ for all $i>0$;
\item \label{prop:LmIrregular:2}
if $bR^*a$, then $b R^i a$ for all $i>0$.
\end{enumerate}
\end{proposition}
\begin{proof}
Since $C$ is $m$-irregular, there is a subset $D=\{a_0,\ldots, a_m\}$ of $C$ with $a_0R a_1 R\ldots R a_{m} R a_0$.

First, observe that for every $c\notin D$, we have
\begin{equation}\label{eq:cycle}
\EE d\in D (dRc)\, \;\Rightarrow \;\AA d\in D\, (dRc).
\end{equation}
Indeed, if $dRc$, then consider the point $d'$ that follows $d$ in the cycle $D$: then $d'R^{m}dRc$, $d'\neq c$, and so $d'Rc$. This yields that $c$ is accessible from any point in $D$.

Now we claim that
\begin{equation}\label{eq:ref:manyways-0}
\AA d\in D \; (R^*(d){\setminus}C\subseteq R(d)).
\end{equation}
Let $d R^* c$ for some $d\in D$, $c\notin C$.
By $m$-transitivity of $R$, we have $d R^l c$  for some $l$ with $0<l\leq m$.
For some $d'\in D$, we have $d' R^{m+1-l}d$, so $d'R^{m+1} c$, and  so $d'Rc$. From \eqref{eq:cycle}, $dRc$, which proves
\eqref{eq:ref:manyways-0}.

Now we show that
\begin{equation}\label{eq:ref:manyways}
\AA d\in D \; (R^*(d){\setminus}C\subseteq R^i(d) \text{ for all positive $i$}).
\end{equation}
Let $d R^* c$ for some $d\in D$, $c\notin C$.
Consider $i>0$. We have $i=q(m+1)+r$ for some $q,r\geq 0$, $r\leq m$.
Assume $r>0$. Consider $d'\in D$ such that $dR^{r-1}d'$. By \eqref{eq:ref:manyways-0}, we have
$d'Rc$. Hence: $dR^{q(m+1)}dR^{r-1}d'Rc$, 
so $dR^ic$. If $r=0$, then $q>0$. Let $d'$ be the point that precedes
$d$ in the cycle $D$. We have $dR^{(q-1)(m+1)}d R^m d'$ and  $d' R c$ by \eqref{eq:ref:manyways-0}, so $dR^ic$.

This proves the first statement of the proposition  for the case when $D=C$.

Assume that $aR^* b$ for $a\in C{\setminus}D$.
Consider $d\in D$. 
Since $D\ne C$ and $a\ne d$,  by Proposition \ref{prop:irregAiML}  we have that $R\cup Id_C$ is universal on $C$ and that $aRd$.
Let us show that $aR^ib$ for all $i>0$.
Assume $i>1$. We have $dR^{i-1} b$ by \eqref{eq:ref:manyways}; so $aR^ib$. For the case $i=1$, 
observe that we have $d R^m b$ by \eqref{eq:ref:manyways}, so $aR^{m+1} b$; since $b\notin C$, $a\neq b$, and hence $aRb$.

This proves the first statement of the proposition. The second statement follows from the first and
the observation that $(X,R^{-1})$ is an $\Lm$-frame.
\end{proof}

\begin{proposition}\label{prop:LMregularC}
If $C$ is an $m$-regular cluster in an $\Lm$-frame $(X,R)$, 
then $R^{nm+1}(a)\subseteq R(a)$ for all $a\in C$, $n<\omega$.
\end{proposition}
\begin{proof}
Let $aR^{mn+1} b$. To see that $aRb$, consider two cases: the case $b\in C$, which follows from regularity of $C$;
the case $b\notin C$, which follows from Proposition \ref{prop:Lm-diffClusters}.
\end{proof}

\hide{
\begin{definition}\label{def:eq-clusters}
For a model $M=(F,\theta)$ and a set of formulas $\Delta$,
we define the relation $\approx_{M,\Delta}$ on $F$.

First, for a cluster $C$ in $F$, we define the algebra $A(C)$ as follows. Let $F_C$ be the restriction of
$F$ to $C$;
consider the algebra $\Alg(F_C)$ of the frame $F_C$.
 $A(C)$ is defined as the subalgebra of $\Alg(F_C)$ generated by the set $$\{\valext(\vf)\cap C\mid \vf\in \Delta\}.$$

We define ${\approx}_{M,\Delta}$ as the set of pairs $(a,b)$ such that
there is an isomorphism $f:A([a]_R)\to A([b]_R)$ of algebras such that
for all $\vf\in\Delta$,
\begin{equation}\label{eq:approx-def}
f(\valext(\vf)\cap [a]_R )=\valext(\vf)\cap [b]_R.
\end{equation}
\end{definition}

Clearly, $\approx_{M,\Delta}$ is an equivalence.

\begin{proposition}\label{prop:finiIndexLFClusters}
Assume that the logic of the set of clusters in $F$ is locally finite. Then for any model $M$ on $F$,
for any finite set $\Delta$ of formulas,
$\approx_{M,\Delta}$ is an equivalence of finite index. \ISh{inconsistency in the style of notation}
\end{proposition}
\begin{proof} Let $L$ be the logic of the clusters in $F$, $k=|\Delta|$.
Every algebra $A(C)$ is $k$-generated, so it is a homomorphic image of the $k$-generated free algebra
of $L$. The latter algebra is finite. Hence, every set of pairwise non-isomorphic $k$-generated $L$-algebras is finite, so the set of
$\approx_{M,\Delta}$-classes is.

\end{proof}

}


\begin{proposition}\label{prop:sel:chain} Let $F=(X,R)$ be an $\Lm$-frame,
$a_m R^* a_{m-1} R^* \ldots R^* a_0$.  
Assume that $[a_i]_R\neq [a_j]_R$ for some $i,j\leq m$, or 
all $a_i$ belong to an $m$-regular cluster. 
Then $R(a_i)\subseteq  R(a_j)$ for some $i<j\leq m$.
\end{proposition}
\begin{proof}
We firstly assume that not all $[a_i]_R$ are $m$-regular. In this case, there are distinct clusters among them. Then for some $i<j\leq m$ we have: $[a_i]_R\neq [a_j]_R$, 
  and one of them is $m$-irregular.   
Assume that $[a_j]_R$ is $m$-irregular.   
By Proposition \ref{prop:LmIrregular}.\ref{prop:LmIrregular:1}, 
$a_j R^{m} a_i$. Now $R(a_i)\subseteq R(a_j)$ 
by Proposition \ref{prop:Lm-diffClusters}.
The case when $[a_i]_R$ is $m$-irregular     follows similarly from Proposition \ref{prop:LmIrregular}.\ref{prop:LmIrregular:2}.

Now assume that all clusters $[a_i]_R$ are $m$-regular.  

We aim to show that for some $i<j\leq m$,  
\begin{equation}\label{eq:sel:mult}
\text{$a_j R^k a_i$ with $k$ a multiple of $m$}.
\end{equation} 
For $s<m$, fix an $l_s$ such that $a_{s+1} R^{l_{s}} a_s$.  
For $0<j\leq m$, put $k_j=\sum_{s< j}l_s$. If some $k_j$ is a multiple of $m$, then 
\eqref{eq:sel:mult} holds for $i=0$.  
Otherwise, $k_i\equiv k_j \mod m$ for some $i<j$ (there are $m-1$ non-zero remainders for 
$m$ numbers $k_1$,\ldots, $k_m$). So  
$k=\sum_{i\leq s<j}l_s$ is a multiple of $m$. We have $a_j R^k a_i$, so \eqref{eq:sel:mult} holds in this case as well.
  
By Proposition \ref{prop:LMregularC}, $R^{k+1}(a_j)\subseteq R(a_j)$. If $a_i R a$, then $a_j R^{k+1} a$, and so $a_j R a$.
\end{proof}

\subsection{Weak $m$-closure}


\begin{definition}\label{def:m-closure-alt}
Let $m>0$. For a frame $(X,R)$, define $R^{[m]}_n$ by recursion on $n$:
$$
\textstyle
R^{[m]}_0=R,\quad R^{[m]}_{n+1}=
\left(\bigcup_{i\leq n} R^{[m]}_{i}\right)^{m+1}{\setminus} Id_X.$$
Put
$$
\textstyle
R^{[m]}= \bigcup_{n\in \omega} R^{[m]}_n.
$$
The frame $(X,R^{[m]})$ is called the {\em weak $m$-closure of $(X,R)$}.
\end{definition}


\smallskip 

\begin{proposition}\label{prop:m-closure} Let $(X,R)$ be a frame. Then
$R^{[m]}$ is the smallest relation $S$ on $X$ such that $S$ contains $R$ and
$(X,S)$ is an $\Lm$-frame.
\end{proposition}
\begin{proof}
Clearly, $R^{[m]}$ contains $R^{[m]}_0=R$. Assume
that $a  (R^{[m]})^{m+1} b$ and $a\neq b$.
\ISh{-- I changed this place -- got  unclear, if it was correct}
Then  for some $n_0,n_1,\ldots, n_{m}$ we have 
$a R^{[m]}_{n_0}\circ R^{[m]}_{n_1} \circ \ldots \circ R^{[m]}_{n_m} b$. 
Then  for  $n=\max\{n_0,\ldots,n_m\}$,  $a R^{[m]}_{n+1} b$.  
Hence $a  R^{[m]} b$, and so
the $m$-closure $(X,R^{[m]})$ of $(X,R)$ validates \eqref{eq:m-property}.

If $S\supseteq R$ and $(X,S)$ is an $\Lm$-frame, then
that $S$ contains all $R^{[m]}_n$ follows by an easy induction on $n$.
\end{proof}

\begin{remark*}
For the case $m=1$, the weak   
closure can be defined in the following way:
\begin{equation}\label{eq:1closure}
R \cup (R^+ \setminus Id_X).
\end{equation}
Proposition \ref{prop:m-closure}  for the case $m=1$ was given in  \cite{esakia2001weak}. 

The expression \eqref{eq:1closure} has a simpler form than the one given in Definition \ref{def:m-closure-alt}.
Notice however that an immediate generalization  of \eqref{eq:1closure} for the case $m>1$, 
namely the relation 
$R \cup (S \setminus Id_X)$, 
where $S$ is the $m$-closure defined in Section \ref{sec:Filtr}, does not satisfy the closure conditions of Proposition \ref{prop:m-closure}. 

For example, let $m=2$. Consider a three-element irreflexive antisymmetric cycle $(X,R)$. This is an $\logicts{L}_2$-frame, so $R^{[2]}=R$. 
However, the $m$-closure $S$ of $R$ is $X\times X$, and so $R\cup (S \setminus Id_X)$ is not equal to $R$. 
\end{remark*}

\medskip

The following  fact  is immediate from Definition \ref{def:m-closure-alt}.

\smallskip 

\begin{proposition}\label{prop:m-closure-irrelf} Let $F$ be a frame, $m<\omega$. A point $a$ of $F$ is reflexive in $F$
iff $a$ is reflexive in the weak $m$-closure of $F$.
\end{proposition}

\smallskip

\begin{proposition}\label{prop:m-closure-reduce}
$R^{[m]}\subseteq \bigcup_{l\in \omega} R^{ml+1}$.
\end{proposition}

\begin{proof}
By induction on $n$, we show that 
$R^{[m]}_n \subseteq \bigcup_{l\in \omega} R^{ml+1}$
for all $n<\omega$. 
The basis is trivial.
Let $n>0$.
By the induction hypothesis,
for some $l_0,\ldots, l_{m}$ we have:
$R^{[m]}_{n}\subseteq
R^{ml_0+1}\circ R^{ml_1+1}\circ \ldots \circ R^{ml_m+1}=
R^{mk+1}$, where $k= \sum_{i\leq m}l_i+1$.
\end{proof}
\hide{
\begin{proof}
    Let $aR^{[m]} b$ then $a R^{[m]}_k b$ for some $k$. We prove by induction on $k$ that $(a,b) \in \bigcup_{n\in \omega} R^{mn+1}$.

    For $k = 0$ it is obvious.

    Now let $a R^{[m]}_{k+1} b$ then $a S^{m+1} b$ for $S=\bigcup_{i\le k} R^{[m]}_{i}$:
    $$
    a R^{[m]}_{i_1} a_1 R^{[m]}_{i_2} \ldots R^{[m]}_{i_m} a_m R^{[m]}_{i_{m+1}} b.
    $$
    By induction hypotheses
    $$
    a R^{j_1m+1} a_1 R^{j_2m+1} \ldots R^{j_m m+1} a_m R^{j_{m+1} m + 1} b.
    $$
    Hence,
    $$
    a R^{(j_1 + j_2 + \ldots+j_m)m+m+1} b \iff a R^{(j_1 + j_2 + \ldots+j_m +1)m+1} b. \qedhere
    $$
\end{proof}
\ISh{I believe that without points this would look much simpler; the fact is trivial, in fact:}}

\subsection{\texorpdfstring{$\Lm$-frames}{Lm-frames} admit filtration}
The proof is based on the following two components. 

The first is the existence of the weak $m$-closure, which  will be applied to the minimal filtered relation.

The second crucial component is the local tabularity of the logic of $\Lm$-clusters: it will be used to define the equivalence.

A logic $\vL$ is {\em locally tabular},
if each of its finite-variable fragments contains only
a finite number of pairwise nonequivalent formulas.
In algebraic terms, it means that each finitely generated free (or Lindenbaum-Tarski) algebra of $L$ is finite.

\smallskip 

\begin{theorem}\label{thm:LmclastersLF}\cite{LocalTab16AiML}
For every $m>0$, the logic of $\Lm$-clusters is locally tabular.
\end{theorem}

\medskip

\begin{theorem}\label{thm:ADF-Lm}
For every $m>0$,
the class of $\Lm$-frames admits filtration.
\end{theorem}
\begin{proof}
Let $M=(F,\theta)$ be a model over an $\Lm$-frame $F=(X,R)$, $\Gamma$ a finite $\Sub$-closed set
of formulas. Let $\Delta=\Gamma\cup \{\Di^i\vf \mid i\leq m\, \&\, \Di\vf \in \Gamma\}$.

For a point $a$ in $M$, let
$\Delta(a)=\{\vf\in \Delta \mid M,a\mo\vf \}.$
On $X$, put
$$a\sim b \tiff \Delta(a)=\Delta(b).$$
Hence, $\sim$ is the equivalence $\sim_\Delta$.
In \cite{Gabbay:1972:JPL}, this equivalence was used to construct filtrations for the logics
given by the axioms $\Di^{m+1}p\imp \Di p$;
for the case of formulas
$\Di^{m+1}p\imp \Di p\vee p$, we need a finer equivalence.

%
For a cluster $C$ in $F$, we define the modal algebra $A(C)$ as follows. 
Consider the algebra $\Alg(F\restr C)$ of the frame $F\restr C$.
 $A(C)$ is defined as the subalgebra of the modal algebra $\Alg(F\restr C)$ generated by the set $$\{\valext(\vf)\cap C\mid \vf\in \Delta\}.$$
Let $k=|\Delta|$.  By Theorem \ref{thm:LmclastersLF}, the logic $L$ of $\Lm$-clusters is locally tabular, and
so the $k$-generated free algebra $A$
of $L$ is finite; let $N$ be the size of $A$.
Every algebra $A(C)$ is $k$-generated, so it is a homomorphic image of $A$. Hence,
for every cluster $C$ in $F$, we have:
\begin{equation}\label{eq:epi-size}
  \text{ the size  of $A(C)$ is not greater than $N$.}
\end{equation}

Let $C,D$ be clusters in $F$ and let $f:A(C)\to A(D)$ be an isomorphism of algebras;
we say that $f$ is a {\em $\Delta$-morphism}, if 
for all $\vf\in\Delta$,
\begin{equation}\label{eq:approx-def}
f(\valext(\vf)\cap C)=\valext(\vf)\cap D.
\end{equation}
So $\Delta$-morphism  is an isomorphism of modal algebras endowed with $k$ constants. 

We define the relation $\approx$ on $X$ as the set of pairs $(a,b)$ such that
there exists a $\Delta$-morphism $f:A([a]_R)\to A([b]_R)$.
\IS{Such an isomorphism is unique: if two homomorphisms coincide on generators, they are equal.} 
Clearly, $\approx$ is an equivalence. 
From $\eqref{eq:epi-size}$, it follows  that  
there are only finitely many non-isomorphic algebras of the form $A([a]_R)$; hence, $X/{\approx}$ is finite.\IShLater{It is exactly  about $A([a]_R)$,  rather about its expansions with constants. Still signature is finite, so need not care too much.}

\hide{
\footnote{
$\Delta$-morphism  can be viewed as an isomorphism of relational structures in a  signature with a binary relation symbol and finitely many unary predicate symbols. Namely, for a cluster $C$, define $S(C)$
as the structure with the binary relation $R\restr C$, and subsets induced on $C$ by the formulas in $\Delta$.
Then $a\approx b$ means that $S([a]_R)$ and  $S([b]_R)$ are isomorphic in the usual model-theoretic sense. \IShLater{DC} \IS{This is nonsense! $C$ and $D$ can be of different Size!}
}
}

\IS{May 2025. Seems like this is the same as the intersection of $\approx$ and $\sim$ (idea we started with?). - No, why?!}
Let $[a]_\at$ be the atom\footnote{Atoms are considered in the standard Boolean sense, as minimal non-zero elements of the algebra.} of the algebra
$A([a]_R)$ that contains $a$; since $A([a]_R)$ is finite,
$[a]_\at$ is properly defined.
We put
$a\equiv b$, if $[b]_\at=f([a]_\at)$ for a $\Delta$-morphism $f:A([a]_R)\to A([b]_R)$.
Clearly, $\equiv$ is an equivalence that refines $\approx$. 
If $E$ is the $\approx$-class of $a$, then the number of $\equiv$-classes containing in $E$ is finite, since the number of atoms in $A([a]_R)$ is finite.
Since $X/{\approx}$ is finite, $X/{\equiv}$ is finite as well. 
\IShLater{May-June 25: OK, this is fine: if there are more elements in $\approx$-class  than the number of atoms, 
two are in one $\equiv$-class. Rechecked on paper (p. 15). Do we want to  have more details?}


The quotient $X/{\equiv}$ will be the carrier of the filtration.
For $a\in X$, let $[a]$ denote the $\equiv$-class of $a$.

\begin{lemma}\label{lem:proof-refines}
If $M,a\mo \vf$ for some $\vf\in \Delta$, and $a\equiv b$, then $M,b\mo \vf$.
\end{lemma}
\begin{proof}
Let $a\equiv b$.
For a $\Delta$-morphism $f$, we have $f:A([a]_R)\to A([b]_R)$ and $[b]_\at=f([a]_\at)$.
Assume that $M,a\mo \vf$ for some $\vf \in \Delta$.
Then $[a]_\at\subseteq \valext(\vf)\cap [a]_R $, and so
$f([a]_\at)\subseteq f(\valext(\vf)\cap [a]_R )$. Then
$$b\in [b]_\at=f([a]_\at)\subseteq f(\valext(\vf)\cap [a]_R )=\valext(\vf)\cap [b]_R;$$
the latter step is given by \eqref{eq:approx-def}. So $M,b\mo \vf$.
\end{proof}

It follows from this lemma that if $a\equiv b$, then $\Delta(a)=\Delta(b)$. So we have:
\begin{equation}
 \equiv\; \subseteq \;\sim.
\end{equation}

Let $\fm{R}$ denote the minimal filtered relation on $X/{\equiv}$:
$$[a]\fm{R}[b] \tiff \EE a'\equiv  a\, \EE b'\equiv b\, (a'Rb').$$

 Consider the frame $(X{/}{\equiv}, \fm{R})$ and its weak $m$-closure
$\ff{F}=(X{/}{\equiv}, \fm{R}^{[m]})$, introduced in Definition \ref{def:m-closure-alt}.
By Proposition \ref{prop:m-closure},
$(X{/}{\equiv}, \fm{R}^{[m]})$ is an $\Lm$-frame,  and
$\fm{R}^{[m]}$ contains $\fm{R}$. So we only need to check that
$\fm{R}^{[m]}$ is contained in the maximal filtered relation $R_{\equiv, \Gamma}$. 
This proof is more convoluted than the one for $\wK4$, and we start with some auxiliary statements.

\begin{lemma}\label{lem:proofLMAF}
Assume that $C,D$ are clusters in $F$,
$f:A(C)\to A(D)$ is an isomorphism, $a,b\in C$,
$a'\in f([a]_\at)$. Then for any $n\geq 0$, if $aR^n b$, then
$a'R^n b'$ for some $b'\in f([b]_\at)$.  
\end{lemma}
\begin{proof}
Let $\Di_1,\Di_2$ be the modal operations in $A(C)$ and $A(D)$\IShLater{test}, respectively.
\hide{
The value of a formula $\vf$ (considered as an algebraic term) in $A$ and $A'$ are denoted
as $\valext(\vf)$ and $\valext(\vf)'$, respectively. }
Consider the element $[a]_\at\cap \Di_1^n [b]_\at$ of $A(C)$. This set is non-empty because $a R^n b$.
Since $[a]_\at$ is an atom,
we have $[a]_\at\subseteq \Di_1^n [b]_\at$, and
so $f([a]_\at) \subseteq f(\Di_1^n [b]_\at) = \Di_2^n f([b]_\at)$. Hence $a'R^n b'$ for some $b'\in f([b]_\at)$.
\end{proof}

\begin{lemma}\label{lem:same-cluster}
If $a\approx a'$, then $[a]$ and $[a']$ belong to the same cluster in $(X{/}{\equiv}, \fm{R})$:
$[a]\fm{R}^*[a']$ and $[a']\fm{R}^*[a]$.
\end{lemma}
\begin{proof}
Let $f$ be a $\Delta$-morphism between $A([a]_R)$ and $A([a']_R)$.
For some
 $b\in [a]_R$ we have $f([b]_\at)=[a']_\at$, and for some $c\in [a']_R$, 
we have $f([a]_\at)=[c]_\at$.  
Since $a,b$ are in the same $R$-cluster, we have $aR^n b$ for some $n\geq 0$. By Lemma
\ref{lem:proofLMAF}, for some $b'\in [a']_\at$ we have $c R^n b'$.
So $[c] \fm{R}^n [b']$.
Since $[c]=[a]$ and $[b']=[a']$, we have $[a]\fm{R}^n[a']$, and so $[a]\fm{R}^*[a']$.

Since $\approx$ is symmetric, we have $[a']\fm{R}^*[a]$.
\end{proof}

For $n>0$, let $\rem{n}= ((n-1) \mod m)+1$, where $\mod$ is the remainder operation; also put $\rem{0}=0$.

Let $C$ be a cluster in $F$, $A = A(C)$, $\Di_A$ the modal operation on
$A$.
We say that $A$ is {\em $m$-regular}, if $\Di_A^{m+1}U\subseteq \Di_A U$ for each $U$ in $A$. 
Otherwise, $A$ is {\em $m$-irregular}. 
Clearly, if $F\restr C$ is $m$-regular, then $A(C)$ is $m$-regular, since    $A(C)$ is a subalgebra of $\Alg(F\restr C)$.\footnote{
An equivalent way to define the $m$-regularity of 
$A(C)$ is to consider the equivalence induced on $C$ by the elements of $A(C)$, and take the  minimal filtered relation on the quotient. Then $m$-regularity of $A(C)$ is equivalent to the validity of 
$\Di^{m+1}p\imp \Di p$ in the resulting quotient-frame.
}

\begin{lemma}\label{lem:filtr-regularPath}
Assume that $[a_n] \fm{R}  [a_{n-1}] \fm{R} \ldots \fm{R} [a_{0}]$ and
for all $i\leq n$, $A([a_i]_R)$ is $m$-regular.
If $\Di\vf\in \Gamma$ and $M,a_0\mo \vf$, then $M,a_i\mo \Di^{\rem{i}}\vf$ for all $i\leq n$.
\end{lemma}
\begin{proof}
This proof essentially follows the argument provided in \cite{Gabbay:1972:JPL} for the logic given by the axiom $\Di^{m+1}p\imp \Di p$.

By induction on $i$. The basis is given.
Let $i>0$. We have $aRb$ for some $a\in [a_i]$, $b\in [a_{i-1}]$. By induction hypothesis, $M,a_{i-1}\mo \Di^{\rem{i-1}}\vf$. We have $\Di^{\rem{i-1}}\vf\in \Delta$, because ${\rem{i-1}}\leq m$; hence
$M,b\mo \Di^{\rem{i-1}}\vf$. So $M,a\mo \Di^{\rem{i-1}+1}\vf$.
\IShLater{Rewrite for $i$, $i+1$}

From the definition of $\rem{i}$, it follows that for $i>0$
\hide{
$$
\rem{i} = \begin{cases}
    1, &\text{if }  i-1  \text{ is a multiple of $m$};\\
    \rem{i-1} + 1, &\text{otherwise.}
\end{cases}
$$
}
$$
\rem{i-1}=m \text{\quad  iff \quad   $i-1$  is a multiple of $m$ and $i>1$}.
$$

Consider two cases.

\smallskip 
First, assume 
that $\rem{i-1}<m$. Then $i-1$ is not a multiple of $m$ or $i=1$. In either case, 
$\rem{i-1}+1=\rem{i}$. 
Hence, $M,a \mo \Di^{\rem{i}} \vf$. 
Since $\rem{i}\leq m$, we have 
$\Di^{\rem{i}}\vf\in \Delta$. It follows that $M,a_i\mo \Di^{\rem{i}}\vf$.

\smallskip 
Now assume $\rem{i-1}=m$. In this case $\rem{i}=1$.  
We have $M,a\mo \Di^{m+1}\vf$, 
and so $aR^{m+1} c$ for some $c\in\valext{(\vf)}$. 

If $c\neq a$, we have $aR c$, and so $M,a\mo  \Di\vf$.

Suppose $a=c$. 
Put $U=\valext(\vf)\cap [a]_R$. 
Since $a\equiv a_i$, we have $a\approx a_i$, and so 
$A([a]_R)$ and $A([a_i]_R)$ are isomorphic. 
So $A=A([a]_R)$ is $m$-regular.
We have $a\in U$, and so $a\in \Di^{m+1}_A U$. 
Hence, $a\in \Di_A U$. 
Then $M,a\mo  \Di\vf$ in this case as well. 

\hide{
Suppose $a=c$. Then for some $d\in [a]_R$  we have $aR^mdRc$. 
Put $U=\valext(\vf)\cap [a]_R$. 
Since $a\equiv a_0$, we have $a\approx a_0$, and so 
$A([a]_R)$ and $A([a_0]_R)$ are isomorphic. 
So $A=A([a]_R)$ is $m$-regular.
We have $b\in \Di_A U$, and so $a\in \Di^{m+1}_A U$. 
Hence, $a\in \Di_A U$. 
Then $M,a\mo  \Di\vf$ as well in this case. 
}

Since $a\equiv a_i$, it follows that  $M,a_i\mo \Di\vf$, that is $M,a_i\mo \Di^{\rem{i}}\vf$. 
\hide{

so we need to show that $M,a\mo \Di \vf$. 
We have $M,a\mo \Di^{m+1}\vf$, so $M,c\mo \vf$ for some $c\in R^{m+1}(a)$.
If $c\neq a$, $M,a\mo \Di \vf$. 
Assume $c=a$. 
For some $b\equiv a$, $[b]_R$ is $m$-regular. Consider a $\Delta$-morphism 
$f:A([a]_R)\to A([b]_R)$ with
$f([a]_\at)=[b]_\at$. By Lemma \ref{lem:proofLMAF}, $bR^{m+1}b'$ for some $b'\in [b]_\at$. 
Since $[b]_R$ is $m$-regular, $bRb'$. So $M,b'\mo \Di\vf$. Since $b'\equiv a$, we have $M,a\mo\Di \vf$.


 }
\end{proof}

\begin{lemma}\label{lem:path-weakDi}
Let $\Di\vf \in \Gamma$, $[a]\fm{R}^*[b]$. If $M,b\mo \bigvee_{0\leq i\leq m} \Di^i \vf$,
then  $M,a\mo\bigvee_{0\leq i\leq m} \Di^i \vf$.
\end{lemma}
\begin{proof}

Let $[a] \fm{R}^n [b]$ and $M,b\mo \Di^i \vf$ for $i\leq m$. 
We claim that $M,a \mo \Di^j \vf$ for some $j \le m$. 
By induction on $n$. The base $n=0$ is clear. 
Let $n>0$. Then $[a] \fm{R} [c] \fm{R}^{n-1} [b]$ for some $c$.  
Then for some $a' \equiv a$ and $c'\equiv c$ we have $a'Rc'$. By the induction hypothesis, $M,c' \mo \Di^l \vf$ for some $l\leq m$. It follows that $M,a' \mo \Di^{l+1} \vf$. If $l=m$, $M,a' \mo \Di \vf\vee \vf$. It follows that $M,a' \mo \Di^j \vf$ for some $j\leq m$. Then $\Di^j\vf \in\Delta$, and so 
$M,a \mo \Di^j \vf$.
%
%
%
\end{proof}

We say that $a\in X$ is {\em $\vf$-saturated}, if 
$M,a'\mo \bigwedge_{1\leq i\leq m}\Di^i\vf$ for
all $a'\in[a]_R$. 

\begin{lemma}\label{lem:irreg-filtr-new}
Let $\Di\vf\in\Gamma$. 
Assume that $[a]\fm{R}^* [b]$, and $b$ is $\vf$-saturated. Then $a$ is $\vf$-saturated.
\end{lemma}
\begin{proof}
First, consider the case $[a]\fm{R}[b]$. 

We have $a_0 R b_0$ for some  $a_0\in [a]$, $b_0\in [b]$. 
For $i\leq m$, $\Di^i\vf \in \Delta$, and since \mbox{$b_0\equiv b\in [b]_R$}
and $b$ is $\vf$-saturated, 
we have
$M,b_0\mo \bigwedge_{1\leq i\leq m}\Di^i\vf$.
So we have $M,a_0\mo \Di^i \vf$ for $2\leq i \leq m+1$.
 
Let $a'\in[a_0]_R$. We claim  that 
\begin{equation}\label{eq:Lm-clustertransfer}
  M,a'\mo \bigwedge_{1\leq i\leq m}\Di^{i}\vf.
\end{equation}

Assume $[a_0]_R=[b_0]_R$. Since $b\equiv b_0$, there is a $\Delta$-morphism $f:A([a_0]_R)\to A([b]_R)$; 
so $a' \equiv c$
for some $c\in f([a_0]_\at)\subseteq [b]_R$.
Since $b$ is $\vf$-saturated, for all positive $i\leq m$ we have $[b]_R \subseteq \bar\theta(\Di^i \vf)$ and so $c \in \bar\theta(\Di^i \vf)$,
which implies \eqref{eq:Lm-clustertransfer}.

\smallskip 
Now assume that $[a_0]_R\neq [b_0]_R$. 
Consider two cases.

\smallskip 
In the first case, assume that the cluster frame $F\restr [a_0]_R$ is $m$-regular.  Since
\mbox{$M,a_0\mo \Di^{m+1}\vf$},  we have $M,a_0\mo \Di\vf$  by Proposition
\ref{prop:LMregularC}. By Proposition \ref{eq:Lm-trans-and-subfr},
$F\restr [a]_R$ is $m$-transitive, so we have $a'R^n a_0$ for some $n\leq m$. Then $M,a'\mo \Di^{n+i}\vf$ for $1\leq i\leq m$; by Proposition \ref{prop:LMregularC}, $M,a'\mo \Di^{i}\vf$ for $1\leq  i \leq m$.
\IShLater{Is this clear how to use this Proposition here?}

\smallskip 
In the second case, assume that $F\restr [a_0]_R$ is $m$-irregular. 
Since $[a_0]_R\neq [b_0]_R$, we have \eqref{eq:Lm-clustertransfer} in view of Proposition \ref{prop:LmIrregular}.
Indeed,  $M, b_0 \models \Di\varphi $ and $M, b' \models \varphi $ for some $b'$ such that  $b_0 R b'$. We have  $b' \notin [a_0]_R$, since otherwise $[a_0]_R = [b_0]_R$. Furthermore, because $a' R^* b'$, Proposition \ref{prop:LmIrregular} implies that $a' R^i b'$ for any $i \ge 1$. Therefore, \eqref{eq:Lm-clustertransfer} follows.

\medskip 

So we have \eqref{eq:Lm-clustertransfer} for all $a'\in[a_0]_R$. 
If $d\in[a]_R$, then $d\equiv d'$ for some $d'\in [a_0]_R$; 
so $M,d\mo \Di^i \vf$ for $1\leq i\leq m$. 
\IShLater{Later: improve the exposition of this story.} 
Hence, if $[a]\fm{R}[b]$ and $b$ is $\vf$-saturated, then $a$ is. 

\smallskip 

The case $[a]\fm{R}^*[b]$ follows by straightforward induction on the length of $\fm{R}$-path from $[a]$ to $[b]$.  
\end{proof}

 
\begin{lemma}\label{lem:path-irregular}
Assume that $a \not\approx b$, $A([c]_R)$ is $m$-irregular, and $[a]\fm{R}^*[c]\fm{R}^*[b]$.
If $\Di\vf \in \Gamma$ and $M,b\mo \vf$, then $a$ is $\vf$-saturated. 
\end{lemma}
\begin{proof}

We consider two cases:   $c \not\approx b$ and $c \not\approx a$.

Assume $c \not\approx b$. 

Consider an $\ff{R}$-path from $[c]$ to $[b]$. 
Then $[c]\ff{R}^*[d]$ and $[d']\fm{R}^*[b]$ for some $d,d'$ such that $d R d'$, $c\approx d$, and $d'\not\approx d$: here $[d']$ is the first $\equiv$-class in the path 
such that $c\not \approx d'$. 
Since $A([c]_R)$ and $A([d]_R)$ are isomorphic, the latter algebra is $m$-irregular. It follows that 
the cluster frame $F\restr[d]_R$ is $m$-irregular. Since $d'\not\approx d$,   $[d]_R \neq [d']_R$.
Since $[d']\fm{R}^*[b]$, $M,d'\mo \bigvee_{0\leq i\leq m} \Di^i \vf$ by Lemma \ref{lem:path-weakDi}. 
So $M,d'' \models  \vf$ for some $d''\in R^*(d')$. Clearly, $d''$ is not in $[d]_R$. 
By Proposition \ref{prop:LmIrregular}.\ref{prop:LmIrregular:1}, for any $i> 0$, 
$\Di^i\vf $ is true at any point in $[d]_R$. 
Hence,
$d$ is $\vf$-saturated. By Lemma \ref{lem:irreg-filtr-new}, $a$ is $\vf$-saturated. 

The case $c \not\approx a$ follows analogously via Proposition \ref{prop:LmIrregular}.\ref{prop:LmIrregular:2}.
\end{proof}
 
Let $\Theta$ be a cluster in $(X{/}{\equiv},\fm{R}^{[m]})$. We say that $\Theta$ is {\em homogeneous}, 
if for any $a,b$ such that $[a],[b]\in \Theta$ we have $a \approx b$, and $A([a]_R)$ is $m$-irregular.

\begin{lemma}\label{lem:filtr-same-cluster}
Let $\Theta$ be a homogeneous cluster in $(X{/}{\equiv}, \fm{R}^{[m]})$.
Then on $\Theta$,  $\fm{R}^{[m]}$ is included in $R_{\equiv,\Gamma}$:
for all $[a],[b]\in \Theta$,
\begin{equation}
\text{if }[a]\fm{R}^{[m]} [b], \text{ then } [a] R_{\equiv,\Gamma}[b].
\end{equation}
\end{lemma}
\begin{proof}
First, suppose that for some $[c],[d]\in \Theta$, we have $cRd$, and $[c]_R \neq [d]_R$.

Let $[a]\fm{R}^{[m]} [b]$, $\Di\vf \in \Gamma$, and $M,b\mo \vf$.
We have $[d] \fm{R}^* [b]$, so
by Lemma \ref{lem:path-weakDi}, $M,d\mo \bigvee_{0\leq i\leq m} \Di^i \vf$. Hence, 
$M,d'\mo  \vf$ for some $d'\in R^*(d)$. Notice that $d'$ is not in $[c]_R$.  
Then  by Proposition \ref{prop:LmIrregular}.\ref{prop:LmIrregular:1},
$c$ is $\vf$-saturated. We have $[a] \fm{R}^* [c]$, and by Lemma \ref{lem:irreg-filtr-new}, $a$ is $\vf$-saturated.
In particular,  $M,a\mo \Di \vf$. It follows that 
$[a] R_{\equiv,\Gamma}[b]$. 

\medskip

Now suppose that for all $[c],[d]\in \Theta$ with $cRd$, we have $[c]_R = [d]_R$.

By induction on $n$, we show that if $[a],[b]\in \Theta$, 
$[a]\fm{R}^n[b]$, and $a'\in [a]$, then  there exists $b'$ such that 
\begin{equation}\label{eq:homog}
    a' R^n b' \text{ and }b'\in [b]\cap [a']_R.
\end{equation} 
If $n=0$, then $[a]=[b]$; in this case, we put $b'=a'$. 
Let $n>0$. For some $c$, we have 
$[a]\fm{R} [c]\fm{R}^{n-1}[b]$.  
We have $a_0Rc_0$ for some $a_0\in[a], c_0\in [c]$. Then $[a_0]_R=[c_0]_R$.  
Since $a_0\equiv a'$, there is a $\Delta$-morphism $f: A([a_0]_R)\to  A([a']_R)$ such that 
$f([a_0]_\at)=[a']_\at$.  
By Lemma \ref{lem:proofLMAF}, $a' R c'$ for some $c'\in f([c_0]_\at)$.
It follows that $c'\equiv c_0$, and so $c'\equiv c$. By induction hypothesis, $c'R^{n-1} b'$ for some 
$b'\in[b]\cap [c']_R$. So $a'R^n b'$. Since $[c']_R=[a']_R$, \eqref{eq:homog} follows.

\smallskip 

Assume that $[a]\fm{R}^{[m]}_1 [b]$ for some $[a],[b]\in \Theta$, where $\fm{R}^{[m]}_1$ was given in Definition \ref{def:m-closure-alt}. 
We have $[a]\fm{R}^{m+1} [b]$, and $[a]\neq [b]$. 
By  \eqref{eq:homog}, we have $aR^{m+1} b'$ for some $b'\in [b]\cap [a]_R$. 
Since $[a]\neq [b]$, we have $a\neq b'$.  Hence, $aRb'$, and so $[a]\fm{R}[b]$.

\hide{
We have $[a]=[a_{m+1}] \fm{R} [a_{m}] \fm{R} \ldots \fm{R} [a_0]=[b]$ for some $a_i$, $i\leq m$, and $[a]\neq [b]$.

It follows from Lemma \ref{lem:proofLMAF} that there are points $c_0,\ldots, c_{m+1}$ in the cluster $[a_0]_R$ such that $c_{m+1} R c_{m} R \ldots R c_0$ and $[c_i]=[a_i]$ for all $i\leq m+1$.\IShLater{Many details are missing here: 
sequence of embedding, and reconstruction of relations. Two inductions, quite simple though.}
Since $[a_{m+1}]\neq [a_0]$, $[a_{m+1}]=[c_{m+1}]$, and $[a_0]=[c_0]$, we conclude that $c_{m+1}\neq c_0$. Since 
 $c_{m+1}R^{m+1} c_0$, we have $c_{m+1}Rc_0$, and so $[a]=[c_{m+1}] \fm{R} [c_0]=[b]$. 
 }

So on $\Theta$, $\fm{R}^{[m]}_1$ is included in $\fm{R}$. 
\IShLater{This is elegant! But, in fact, there are more details to discuss: we are working with a restriction on a cluster, not exactly with $\ff{R}$ and its closure.}
By Definition \ref{def:m-closure-alt}, $\fm{R}^{[m]}_{n+1} = (\bigcup_{i\le n}\fm{R}^{[m]}_i)^{[m]}_1$.
So by induction on $n$ it follows that on $\Theta$, 
$\fm{R}^{[m]}_n $ is included in $\fm{R}$ for all $n<\omega$.
Hence,  on $\Theta$, $\fm{R}^{[m]}$ coincides with the minimal filtered relation, and so is contained in the maximal filtered relation.
\end{proof}

\obsolete{

\begin{lemma}\label{lem:filtr-embedding}
Assume $k\geq0$, $[a_0] \fm{R}^{[m]} [a_1] \fm{R}^{[m]} \ldots \fm{R}^{[m]} [a_k]$, and all $a_i, i\leq k$
are $\approx$-equivalent. Then there are points $c_0,\ldots, c_k$ in the cluster $[a_k]_R$ such that
$[c_0] \fm{R}^{[m]} [c_1] \fm{R}^{[m]} \ldots \fm{R}^{[m]} [a_k]$
\end{lemma}

\begin{lemma}[Obsolete]\label{lem:irreg-filtr-old}
Let $\Di\vf\in\Gamma$. Assume that $[a]\fm{R} [b]$, and for all $b'\in[b]_R$, for all $i\leq m$ we have $M,b'\mo \Di^i\vf$.
Then for all $a'\in[a]_R$, for all $i\leq m$ we have $M,a\mo \Di^i\vf$.
\end{lemma}
\begin{proof}
\todo{Repair}
We have $a_0 R b_0$ for some  $a_0\in [a]$, $b_0\in [b]$.
So we have $M,a_0\mo \Di^i \vf$ for $2\leq i \leq m+1$.

Let $a'\in[a]_R$.

Consider two cases.

First, assume that the cluster frame $F\restr [a]_R$ is $m$-regular.  Since
$M,a_0\mo \Di^{m+1}\vf$,  we have $M,a_0\mo \Di\vf$  by Proposition
\ref{prop:LMregularC}. By Proposition \ref{eq:Lm-trans-and-subfr},
$F\restr [a]_R$ is $m$-transitive, so we have $a'R^n a_0$ for some $n\leq m$. Then $a'\mo \Di^{n+i}\vf$ for $1\leq i\leq m$; by Proposition \ref{eq:Lm-trans-and-subfr}, $M,a'\mo \Di^{i}\vf$ for $1\leq  i \leq m$.
\ISh{Is this clear how to use this Proposition here?}

Now, assume that $F\restr [a]_R$ is $m$-irregular. Then $M,a'\mo \Di^{i}\vf$ for $1\leq  i \leq m$ by Proposition \ref{prop:LmIrregular}.
\end{proof}

\begin{lemma}[Obsolete]\label{lem:filtr-irregular-old}
Assume that $[a_n] \fm{R} [a_{n-1}] \fm{R} \ldots \fm{R} [a_{0}]$ and for some
$i\leq n$, the cluster frame $F\restr [a_i]_R$ is $m$-irregular.
Also assume that $a_i\not\approx a_j$ for some $i,j\leq n$.
Then if $\Di\vf\in \Gamma$ and $M,a_0\mo \vf$, then $M,a_n\mo \Di^{i} \vf$ for all $1\leq i\leq m$.
\end{lemma}
\begin{proof}
Let $k=\min\{i\leq n\mid F\restr [a_i]_R  \text{ is $m$-irregular}\}$. So for $i<k$, all
$F\restr [a_i]_R$. By Lemma

Start with Proposition \ref{prop:LmIrregular}, then use Lemma \ref{lem:irreg-filtr}. \todo{details}
\end{proof}

}

Now we show that
$\fm{R}^{[m]}$ is contained in the maximal filtered relation $R_{\equiv,\Gamma}$.

Recall that $\fm{R}\subseteq R_{\equiv,\Gamma}$ by Proposition \ref{prop:min-included-in-max}.


Suppose that $[a] \fm{R}^{[m]} [b]$, $\Di\vf \in \Gamma$, and $M,b\mo \vf$. 
We need to show that 
\begin{equation}\label{eq:LM-max-con}
M,a\mo \Di\vf   .
\end{equation}

By Proposition \ref{prop:m-closure-reduce}, we have
$[a] \fm{R}^{mn+1} [b]$ for some $n$, so
there is
an $\fm{R}$-path
 $[a]=[a_{mn+1}] \fm{R} [a_{mn}] \fm{R} \ldots \fm{R} [a_{0}]=[b]$. 

 \smallskip 
If each $A([a_{i}]_R)$ is $m$-regular for $i\leq mn+1$, 
then $M,a \mo \Di^{\rem{nm+1}} \vf $ by Lemma \ref{lem:filtr-regularPath}; since $\rem{nm+1}=1$, 
\eqref{eq:LM-max-con} follows. 

\smallskip 

 Now assume that $A([a_{i}]_R)$ is $m$-irregular for some $i\leq mn+1$.  Let $c=a_i$,  and consider the following cases.


\noindent{\bf Case 1.} Assume $a\approx b$. 
In this case, $[a],[b]$ belong to the same cluster $\Theta$ in $(X{/}{\equiv}, \fm{R})$ by Lemma \ref{lem:same-cluster}.  

\noindent{\bf Case 2a.} 
Assume that for all $d$ with $[d]\in \Theta$ we have $d\approx b$. 
Since $\Theta$ contains $[c]$ and $A([c]_R)$ is $m$-irregular, $\Theta$ is homogeneous. Then $[a] R_{\equiv,\Gamma}  [b]$ by Lemma \ref{lem:filtr-same-cluster}.

\noindent{\bf Case 2b.} 
Now assume that $[d]\in \Theta$ and $d \not \approx b$ for some $d$.  We have $[a] \fm{R}^*[d]$ and  $[d]\fm{R}^*[c]\fm{R}^*[b]$. The latter 
implies that $d$ is $\vf$-saturated in view of Lemma \ref{lem:path-irregular}.
Then $a$ is $\vf$-saturated by Lemma \ref{lem:irreg-filtr-new}. In particular, $M,a\mo\Di \vf$. 

\noindent{\bf Case 2.} Assume $a \not\approx b$. Now \eqref{eq:LM-max-con} 
 follows from Lemma \ref{lem:path-irregular}.

 \smallskip 
   
Hence, $\fm{R}^{[m]}$ is contained in the maximal filtered relation $R_{\equiv,\Gamma}$. 
\end{proof}

\hide{
\begin{lemma}\label{lem:filtr-sequence-of-embeddings}
Assume $k\geq0$, $[a_k] \fm{R} [a_{k-1}] \fm{R} \ldots \fm{R} [a_0]$, and all $a_i$
are $\approx$-equivalent for $i\leq k$. Then there are points $c_0,\ldots, c_k$ in the cluster $[a_0]_R$ such that
$c_k R c_{k-1} R \ldots R c_0$ and $[c_i]=[a_i]$ for all $i\leq k$.
\end{lemma}
\begin{proof}
Via Lemma \ref{lem:proofLMAF}. \todo{write down the details: sequence of embedding}.

By induction on $k$. For the case $k=0$, put $c_0=a_0$.
Let $k>0$. We have $aRb$ for some $a\in [a_1]$, $b\in [a_0]$ with $aRb$.

By induction hypothesis, there are points
$c_1,\ldots, c_k$ in the cluster $[a_1]_R$ such that
$c_k R c_{k-1} R \ldots R c_1$ and $[c_i]=[a_i]$ for $1\leq i\leq k$.
Since $a\approx b$, there is a $\Delta$-morphism $h:A([a_1]_R)\to A([a_0]_R)$.
\end{proof}

\bigskip
\todo{Remove the obsolete :}


By induction on $n$, we show that
$$\fm{R}^{[m]}_n\subseteq R_{\max}$$
for all  $n<\omega$.
The basis $n=0$ is trivial, since $\fm{R}^{[m]}_0$ is the minimal filtered relation $R_\equiv$.
\improve{Proposition: min is contained in max.}

Suppose that $[a] \fm{R}^{[m]}_{n+1} [b]$. Let $\Di\vf \in \Gamma$, and $M,b\mo \vf$.

Assume $[a]=[b]$. By Proposition \ref{prop:m-closure-irrelf}, $[a]\ff{R}[b]$. Hence $M,a\mo\Di\vf$.

Assume $[a]\neq [b]$.

For some $a_0,\ldots, a_{m+1}, b_0,\ldots,b_m$ we have:
$a_0\equiv a$, $a_{m+1}\equiv b$, and $a_i \fm{R}^{[m]}_{n}  b_i\equiv a_{i+1}$    for $0\leq i\leq m$.

We have
\begin{equation}\label{eq:diamonds-for-b}
\text{ $M,b_i\mo \Di^{m-i}\vf$ for $0\leq i\leq m$.  }
\end{equation}
Indeed, $M,b_m\mo \vf$, and for $i<m$, \eqref{eq:diamonds-for-b} easily follows from the induction hypothesis.

Let $k$ be the smallest $i\leq {m}$ such that $[a_i]_R\neq [b_i]_R$.
\todo{This case should be fine if this set is non-empty; the case when all is mapped in one cluster
should be considered separately and carefully} \todo{Yep! Here is the gap!}
\improve{better}

For $i<k$, fix an isomorphism $f_i$ between $[b_i]$ and $[a_{i+1}]_R$
that maps $[b_i]_\at$ to $[a_i]_\at$; such $f_i$'s exist since $b_i\equiv a_{i+1}$.

\def\a{\mathbf{a}}

For each $i\leq k$, we recursively define a sequence $\sigma_i=\langle \a_0^{(i)},\a_1^{(i-1)},\ldots,\a_i^{(0)} \rangle$
\todo{different letters}
of length $i+1$ such that all elements of the sequence belong to the algebra of the cluster $[a_i]_R$ and are non-empty. We set
$\sigma_0=\langle \a_0^{(0)} \rangle$, where
$\a_0^{(0)}=[a_0]_\at$.  The elements of the $i+1$-th sequence, $0<i<k$, are defined as follows:
$$\langle \a_0^{(i+1)},\a_1^{(i)},\ldots,\a_{i}^{1},\a_{i+1}^{0}  \rangle \;=\; \langle f_i(\a_0^{(i)}),f_i(\a_1^{(i-1)}),\ldots,f_i(\a_i^{(0)}),f([b_{i}]_\at \rangle.$$

Consider $k+1$ elements of $X$: $c_0 \in \a_0^{(k)}, c_1 \in \a_1^{(k-1)},   \ldots, c_k=a_k \in  \a_k^{(0)}$.
\todo{Explain why $a_k\in \a_k^{(0)}$: display an equation}

We have for each $i<k$:
\begin{equation}
\text{ $c_i R^{mn_i+1} c_{i+1}$ for some $n_i$}
\end{equation}
This follows from Proposition \ref{prop:m-closure-reduce}.\todo{No, it does not}

Now consider two cases.

Let $k<m+1$.
Then we have $a_k \Di^{mn+1} b_{k+1} $ for some $n$,
and $[b_k+1]_R\neq [a_k]$.
 and we have  $b_k R^{m-k} d$ for some $d$ such that $M,d\mo \vf$.

We also have $M,c_k\mo \Di^{m+1-k}\vf$ \todo{details: via $a_k$, via $b_k$}.
\end{proof} 
}

It follows that the logics $\Lm$ have the finite model property. 
Since $\Lm$ are canonical logics, 
their tense expansions given by extra axioms \eqref{eq:tenseAxioms} are Kripke complete, and in view of  Theorem \ref{thm:transfer-temporal},  have the finite  model property as well. 

\begin{corollary} 
The logics $\Lm$ and their tense expansions  have the finite model property and are decidable.
\end{corollary}

\section{Selective filtration, maximality property, and subframe pretransitive logics}

While the filtration in the sense of Definition \ref{def:epi}  imply the finite model property for logics in extended languages,
 selective filtration  can be used to obtain the finite model property for extensions of the logic in the same signature.  

The method of selective filtration was proposed in \cite{Gabbay:Sel}.

\begin{definition}\label{def:selective}
Let $M=(X,R,\v)$ and $M_0=(X_0,R_0,\v_0)$ be Kripke models,
$X_0\subseteq X$, $R_0\subseteq R$, and $\v_0(p)=\v(p)\cap X_0$ for
variables.

Let $\Gamma$ be a set of modal formulas.
The model $M_0$ is called a {\em selective filtration of $M$ through $\Gamma$}, if
for every formula $\Di\psi\in\Gamma$, for every $a\in X_0$,
\begin{equation*}\label{eq:selfil}
\text{if } M,a\mo\Di\psi, \text{ then there exists $b$ such that $aR_0 b$ and }
M,b\mo\psi.
\end{equation*}
\end{definition}

The following fact is standard.

\smallskip 

\begin{proposition}[Selective filtration lemma]
Assume that $\Gamma$ is closed under taking subformulas, $M_0$ a selective filtration of $M$ through $\Gamma$.
Then for every $\psi\in\Gamma$, $a$ in $M_0$, we have
$$
M,a\mo\psi \text{ iff } M_0,a\mo\psi.
$$
\end{proposition}

\smallskip 

The proof is by a straightforward induction on the length of formula.

\IShLater{Refs. History: Segerberg?}

\subsection{Selective filtration in clusters}

\begin{proposition}\label{prop:sel-cluster-reg}
 Let $F=(X,R)$ be an $m$-regular cluster frame, $V\subseteq X$.
 Then there exists  $U \subseteq V$ such that
 \begin{equation}\label{eq:sel:cluster}
 \text{$|U|\leq m$ and $R^{-1}[V] = R^{-1}[U]$. }
 \end{equation}
\end{proposition}
\begin{proof}
\hide{
Let $Y=R^{-1}[V]$.
Recursively, we define $U_n=\{u_i\}_{i<n}\subseteq V$
and $Y_n=\{a_i\}_{i<n}\subseteq Y$. We put $U_0=Y_0=\emp$.
Let $n>0$. Assume that $a\in Y{\setminus} R^{-1}[U_{n-1}]$;
then there is $u\in R(a)\cap V$; put $u_{n-1}=u$, $a_{n-1}=a$.
Remark that in this case $a_{n-1}\notin Y_{n-1}$ and so $|Y_n|=n$.
If $Y{\setminus} R^{-1}[U_{n-1}]$ is empty, put $U_n=U_{n-1}$, $Y_n=Y_{n-1}$.
}

Let $Y=R^{-1}[V]$.
Recursively, we define $U_n\subseteq V$
and $Y_n\subseteq Y$.
For $n=0$ we put $U_0=Y_0=\emp$.

Let $n>0$. 

If $Y{\setminus} R^{-1}[U_{n-1}] = \emp$ then we put $U_n=U_{n-1}$, $Y_n=Y_{n-1}$. 

If $Y{\setminus} R^{-1}[U_{n-1}] \ne \emp$, then there exist $a_{n}\in Y{\setminus} R^{-1}[U_{n-1}]$ and $u_{n}\in R(a_{n})\cap V$. We put $U_n = U_{n-1} \cup \set{u_{n}}$ and $Y_n = Y_{n-1} \cup \set{a_{n}}$. Notice that in this case  $|Y_n|=n$.

\smallskip


Put $U=U_m$. 
If $|Y_m|<m$, then 
$Y{\setminus} R^{-1}[U]$ is empty due to the construction, and so \eqref{eq:sel:cluster} holds.

Assume that $|Y_m|=m$. 
We show that in this case $R^{-1}[U]$ contains all points in the cluster.
Clearly, $Y_m\subseteq R^{-1}[U]$. 
Assume that  $a\in X{\setminus}Y_m$. 
By the construction, if $1\leq i<j\le m$, then $R(a_j) \not\subseteq R(a_i)$: indeed, 
 $u_{j}\in R(a_j)$, but 
$a_j \notin R^{-1}[U_{j-1}]$, and $a_i \in U_{j-1}$.
By Proposition \ref{prop:sel:chain},  $R(a_i)\subseteq R(a)$ for some $i<m$. 
We have 
$u_i\in R(a_i)$,  and so  $u_i\in R(a)$. Hence, $a\in R^{-1}[U]$.

\hide{
For $1\leq i<m$, let $l_i$ be the least $l$ such that $y_{i+1} R^l y_{i}$; we have $l_i\leq m$ by $m$-transitivity of $F$.
Let also $l_m$ be the least $l$ such that $y R^l y_{m-1}$.

For $1\leq a\leq b\leq m$, set $k_{ab}=\sum_{a\leq i\leq b} l_i \mod m$,
and consider the set $S=\{k_{am}\}_{1\leq a\leq m}$.
We claim that $k_{am}=0$ for some $a$. Assume not. In this case $|S|\leq m-1$, and hence
$k_{am} = k_{bm}$ for all $a<b\leq m$.
Then $\sum_{a\leq i\leq {b-1}}l_i$ is a multiple of $m$, and so
$y_{b} R^{mn} y_a R u_a$ for some $n$, and so $y_{b} R u_a$. This contradicts the construction.
\ISh{(later) or the assumption; improve the logic}
Hence $k_{am}=0$ for some $a$. Then
$y R^{mn} y_a R u_a$, and so $y R u_a$. This proofs \eqref{eq:sel:cluster}. 
}
\end{proof}

\begin{proposition}\label{prop:sel:cluster}
 Let $F=(X,R)$ be an  $\Lm$-cluster frame, $V\subseteq X$.
 Then there exists $U \subseteq V$ such that $|U|\leq m+1$ and $R^{-1}[V] = R^{-1}[U]$.
\end{proposition}
\begin{proof}
If $F$ is regular, it follows from Proposition \ref{prop:sel-cluster-reg}.
If $R\cup Id_X = X\times X$, it is trivial that such $U$ with $|U|\leq 2$ exists.
By Proposition \ref{prop:irregAiML},  in the only remaining case, $F$ is a cycle of size $\leq m+1$, and we put $U=V$.
\end{proof}

Let $\lngth{\vf}$ denote the number of subformulas of $\vf$.

\smallskip 
\begin{theorem}
Let $m>0$. Then a formula is satisfiable in an $\Lm$-cluster iff 
it is satisfiable in an $\Lm$-cluster of size at most $2\lngth{\vf}m$. 
\end{theorem}
\begin{proof}
Readily follows from Proposition \ref{prop:sel:cluster}: for a model $M=(F,\val)$ on an $\Lm$-cluster and a formula $\vf$,  for each its subformula $\psi$, consider the set $V=\valext(\psi)$
and the corresponding set $U$ satisfying \eqref{eq:sel:cluster}; put $U=U_\psi$. 
Then the restriction $M_0$ of $M$ to the union of these sets is a selective filtration:
indeed, if $M,a\mo \Di \psi$ 
for $a$ in $M_0$ and a subformula $\Di\psi$ of $\vf$, then $aRb$ for some $b\in U_\psi$. 
So $M_0$ is the  required countermodel.
\end{proof}

\begin{corollary}
For $m>0$, the satisfiability problem on $\Lm$-clusters is in $\mathrm{NP}$.
\end{corollary}

\subsection{Maximality lemma and subframe pretransitive logics}

Consider a frame $(X,R)$ and its subset $V$.
We say that $a\in V$ is a {\em maximal element of $V$},
if for all $b\in V$,  $aR^* b$ implies $b R^* a$.

An important property of canonical transitive frames is that every non-empty definable subset has a maximal element \cite{Fine85}. This property transfers for the pretransitive case as well.  The following proposition generalizes \cite[Proposition 6]{gliv}.

\smallskip 

\begin{proposition}[Maximality lemma]\label{prop:max-general}
Suppose that $\frF=(X,R)$ is the canonical frame of a pretransitive $\vL$,
$\Psi$ is a set of formulas. If $\{a\in X\mid \Psi\subseteq a\}$ is non-empty (that is, $\Psi$ is $L$-consistent),
then it has a maximal element.
\end{proposition}
\begin{proof}
Consider the skeleton $\ff{F}=(\ff{X},\leq_R)$ of $F$.
Let $V=\{a\in X\mid \Psi\subseteq a\}$,  
$\ff{V}=\{C\in\ff{X}\mid C\cap V\neq \emp\}.$

Let $\Sigma$ be a chain in $\ff{V}$, $\Sigma_0=(\bigcup\Sigma)\cap V$.
Consider the family $$\clU=\{R^*(a)\cap V\mid a\in \Sigma_0\}.$$
It is straightforward that this family has the finite intersection property: we have 
$\bigcap \clU_0\neq \emp$ 
for every finite $\clU_0\subseteq \clU$. 

In a pretransitive canonical frame,
$R^*(a)=\bigcup_{i\leq m} R^i(a)$ for some $m<\omega$. Each $R^i(a)$ is defined
by the set $\Psi_i=\{ \vf \mid \Box^i \vf\in a\}$: $a R^i b$ iff $\Psi_i\subseteq b$,
see, e.g., \cite[Proposition 5.9]{CZ}.
It follows that each $U\in \clU$ is closed in the Stone topology $\tau$ on $F$ (recall that $\tau$ is given by the base 
$\{ \valext(\vf) \mid \vf \text{ is a formula}\}$).

It is well-known that $\tau$ is compact; see, e.g., \cite[Theorem 1.9.4]{gold:math-mod93}.  Hence
$\bigcap \clU$ is non-empty. Consider $a\in \bigcap \clU$. Then the cluster $[a]_R$ is an upper bound of the chain $\Sigma$.

By Zorn lemma, $\ff{V}$  has a maximal element $C$ in the poset $\ff{F}$.
By the definition of $\ff{V}$, $V\cap C$ is non-empty. Consider $b\in V\cap C$. Due to the construction, $b$ is a maximal element of $V$.
\end{proof}

\begin{corollary}\label{cor:max}
Suppose that $\frF=(X,R)$ is the canonical frame of a pretransitive logic.
Let $a\in X$, $\vf\in b$ for some $b$ in $R(a)$. Then $R(a)\cap\{b\mid \vf\in b\}$
has a maximal element.
\end{corollary}
\begin{proof}
Let $\Psi=\{\vf\}\cup\{\psi\mid \Box\psi\in a\}$ and apply 
Proposition \ref{prop:max-general}.
\end{proof}

Maximality is a very useful tool. Similarly to the transitive case, it gives a simple argument for the finite model property of $\wK4$. Namely, let $\vf$ be a $\wK4$-satisfiable formula. 
Consider the set $\Gamma$ of its subformulas. 
For each $\psi\in \Gamma$, take the set $\clV_\psi$ of clusters
in the canonical model $M$ of $\wK4$ where $\psi$ is satisfied in at least one point.
Let $\clU_\psi$ be the set of all maximal clusters in $\clV_\psi$.
The restriction of $M$ to $\bigcup\{\clV_\psi\mid \psi\in \Gamma\}$
will be 
a selective filtration of $M$ through $\Gamma$ and it will have the height
not exceeding the length of $\vf$. This is the crucial step. Clusters can be made
finite according
to Proposition \ref{prop:sel-cluster-reg}; to make the branching in the skeleton finite
is a routine procedure.\footnote{Besides the finite model property, this argument also gives a simple way to prove that $\wK4$ is decidable in $\PSpace$. This was announced in \cite[Section 7]{ShapON_PSPACE05}; proof of the complexity result is given in
\cite{CondSatOnline}. }
Below we give a detailed and more general (and more technical) argument: selective filtration in canonical models of the logics $\Lm$, $m>0$, and their subframe extensions.

\IShLater{Subframe - definition: double check}

\IShLater{More words on results of Fine}


\smallskip

\begin{theorem}\label{thm:selective}
Let $m>0$, $L$  a  subframe canonical logic which extends $\Lm$. Then $L$ has the finite model property.
\end{theorem}
\begin{proof}
Consider the canonical model $M=(X,R,\val)$ of $L$. Assume that $\vf$ is $L$-consistent. So we have $\vf\in a_0$ for some $a_0\in X$.

Recall that for $a\in X$, $[a]_R$ denotes the cluster of $a$.   
Let $\Phi$ be the set of subformulas of $\vf$, $\Gamma$ the set of subformulas of $\vf$ of form $\Di\psi$.
For $a\in X$, put
\begin{align*}
		\Gamma_a &= a\cap \Gamma, \\
		\Gamma_a^\sim &= \{\Di \psi \in \Gamma_a\mid  \exists b\in [a]_R \,(aRb\,\&\,\psi \in b)\},\\
		\Gamma_a^\uparrow &= \Gamma_a {\setminus} \Gamma_a^\sim.
	\end{align*}


For each $a\in X$,  $\Di \psi\in \Gamma_a^\uparrow$, fix a maximal point $c(a,\psi)$ in the set
$R(a)\cap \valext(\psi)$; such a point exists by Corollary  \ref{cor:max}.
By the definitions, 
\begin{equation}\label{eq:propOfMaxc}
(a,c(a,\psi))\in R  \text{ and } [a]_R  \neq [c(a,\psi)]_R.
\end{equation}

Let $\ff{X}$ be the set of clusters in $(X,R)$.  
On $\ff{X}$, for $\psi\in\Phi$, define the {\em $\psi$-parent} relation $\lhd_\psi$:
put $C\lhd_\psi D$, if
$$\EE a\in C ( \Di\psi\in\Gamma_a^\uparrow \,\&\,c(a,\psi)\in D).$$
Put $$\lhd \; =\;\bigcup \{\lhd_\psi\mid \Di\psi\in\Gamma\}.$$
\hide{
Also, put
$\Phi(C)=\{\psi\in \Phi  \mid
\EE a\in C \EE D\in \clU_{n-1}  ( \Di\psi\in\Gamma_a^\uparrow\, \& \,c(a,\psi)\in C \,\& C<_R D)\}.$
} 
From \eqref{eq:propOfMaxc}, we have:
\begin{equation}\label{eq:propOflhd}
   \text{if $C\lhd D$, then $C\neq D$ and $\EE c\,\in C\, \EE d\in D\;  cRd $.}
\end{equation}

\begin{lemma}
Assume that for some clusters we have $C_n\lhd  C_{n-1}\lhd   \ldots \lhd  C_0$.  
Then $n\leq m|\Gamma|+1$. 
\end{lemma}
\begin{proof}
For the sake of contradiction, assume $n> m|\Gamma|+1$. 
Then for some $\psi\in\Phi$, for at least $m+1$ clusters $C_l$ in this chain we have $C_l\lhd_\psi C_{l-1}$.   
Fix $l_m>l_{m-1}>\ldots >l_0>0$ such that 
$C_{l_k}\lhd_\psi C_{l_k-1}$ for all $k\leq m$.
Then there are points  $a_m\in C_{l_m}, \ldots, a_0\in C_{l_0}$ such that 
\begin{equation}\label{eq:lemmaOnParet}
c(a_k,\psi)\in C_{l_k-1}
\end{equation} 
for all $k\leq m$. 
By \eqref{eq:propOflhd},  the clusters of  points $a_k$ form a strictly increasing chain 
in the skeleton of $(X,R)$. 
By Proposition \ref{prop:sel:chain}, $R(a_i)\subseteq  R(a_j)$ for some $i<j\leq m$. 
Hence, $c(a_i,\psi)\in R(a_j)$. On the other hand, by \eqref{eq:lemmaOnParet}, $c(a_j,\psi)R^* c(a_i,\psi)$, while  $c(a_i,\psi)$ and
$c(a_j,\psi)$ belong to different clusters; this contradicts the maximality of 
$c(a_j,\psi)$ in $R(a_j)\cap\valext(\psi)$.
\end{proof}

For each $C\in \ff{X}$, for each $\Di\psi\in \Gamma$, fix a finite set $U(C,\psi)\subseteq C\cap \valext(\psi)$ such that
\begin{equation}
\{a\in C\mid \EE b\in C (aRb\,\&\,\psi \in b)\}\;\subseteq\; R^{-1}[U(C,\psi)];
\end{equation}
such a set exists by Proposition \ref{prop:sel:cluster}.
Put $$U(C)\;=\;\bigcup\{U(C,\psi)\mid  \Di\psi\in \Gamma\}.$$
We have: 
\begin{equation}\label{eq:Gamma-sim}
\text{
If $\Di\psi\in \Gamma_a^\sim$, then $\psi\in b$ for some $b\in U([a]_R)\cap R(a)$.
}
\end{equation}

\medskip 
 
Recursively, we define $Y_n\subseteq X$ and $\clU_n\subseteq  \ff{X}$. 
Put
$$
\clU_0=\{[a_0]_R\}, \quad Y_0=\{a_0\}\cup U([a_0]_R).
$$ 
For $n>0$, put
\begin{eqnarray*}
\clU_n &= &\{[c(a,\psi)]_R\mid a\in Y_{n-1}\,\&\,\Di\psi\in\Gamma_a^\uparrow\},\\
Y_n &= &\{c(a,\psi)\mid a\in Y_{n-1}\,\&\,\Di\psi\in\Gamma_a^\uparrow\} \cup \bigcup\{U(C)\mid C\in \clU_n\}.
\end{eqnarray*}

By induction on $n$, we have:
\begin{eqnarray}
\label{eq:sel:finiteZU}&&\text{$\clU_n$  and  $Y_n$ are finite;}\\
\label{eq:sel:inclusionYandU}&&Y_n\subseteq \bigcup\clU_n;\\
\label{eq:sel:select-sim}&&\text{If $a\in Y_{n}$ and $\Di\psi\in \Gamma_a^\sim$, then $\psi\in b$ for some $b\in Y_n\cap R(a)$;}\\
\label{eq:sel:select-up}&&\text{If $a\in Y_{n}$ and $\Di\psi\in \Gamma_a^\uparrow$, then $\psi\in b$ for some $b\in Y_{n+1}\cap R(a)$;}\\
\label{eq:sel:chain-step}&&
\AA C\in \clU_{n+1}\, \EE D \in \clU_{n} (D \lhd  C);\\
\label{eq:sel:chain}
&&\AA C\in \clU_n  \;([a_0]_R \, \lhd ^n \, C) .
\end{eqnarray}
Finiteness of $\clU_n$ and $Y_n$ is immediate from finiteness of $\Phi$
and $U(C)$ for each cluster $C$; \eqref{eq:sel:inclusionYandU} is straightforward from the definition of $Y_n$. 
To show  \eqref{eq:sel:select-sim}, 
assume that $a\in Y_{n}$ and $\Di\psi\in \Gamma_a^\sim$.  Then  
$[a]_R\in \clU_n$ by \eqref{eq:sel:inclusionYandU}, and so  $U([a]_R) \subseteq Y_n$. Now 
\eqref{eq:sel:select-sim} follows from \eqref{eq:Gamma-sim}.
The statement \eqref{eq:sel:select-up} is given by the first term in the definition of $Y_n$. 
From the definition of $\lhd_\psi$, we have \eqref{eq:sel:chain-step}, which, in turn, implies 
\eqref{eq:sel:chain}.

From the above lemma and \eqref{eq:sel:chain}, it follows that for some $n\leq m|\Gamma|+1$, the set $\clU_{n+1}$ is empty. By \eqref{eq:sel:inclusionYandU},  $Y_{n+1}$ is empty.
From \eqref{eq:sel:select-sim} and \eqref{eq:sel:select-up}, the restriction  $M_0$ of $M$ to $\bigcup_{i\leq n}Y_n$ is a selective filtration of $M$ through $\Phi$. Hence, 
$M_0,a_0\mo \vf$ by Selective filtration lemma. By \eqref{eq:sel:finiteZU}, $M_0$ is finite. 
\IShLater{$M_0=M\restr X_n$ is undefined}

\hide{ 
For $n>0$, we have
\begin{equation}
\AA C\in \clU_n \EE a\in C \EE D\in \clU_{n-1}  \EE \psi \in \Phi ( \Di\psi\in\Gamma_a^\uparrow\, \& \,c(a,\psi)\in C\& C<_R D).
\end{equation}
For $C\in \clU_n$, $n>0$, let 
\begin{equation}
\Phi(C)=\{\psi\in \Phi  \mid 
\EE a\in C \EE D\in \clU_{n-1}  ( \Di\psi\in\Gamma_a^\uparrow\, \& \,c(a,\psi)\in C \,\& C<_R D)\}.
\end{equation}

$\clU={\bigcup{n<\omega}}\clU_n$. 
On $\clU$, define the {\em child} relation $\lhd$: 
$C\prec D$, if for some $n>0$, $D\in \clU_n$, $D\in \clU_{n-1}$, and  
$\EE a\in C \EE \psi \in \Phi ( \Di\psi\in\Gamma_a^\uparrow \,\&\,c(a,\psi)\in C)$
  }

Finally, $M_0$ is based on an $L$-frame, since $L$ is subframe.    
\end{proof}
\IShLater{Add a remark about $k$-canonical  logics}

\IShLater{think out: 
 
Selective filtration is perhaps the simplest argument for the finite model property of $\wK4$.
Also, it gives a counter model of height bounded by a polynomial in the length of the formula.
However, unlike epifiltrations, it does not imply the finite model property for the derived logics.
}

\bigskip


\section{Problems}

\paragraph*{Definable filtrations}

If the equivalence $\sim$ in Definition \ref{def:epi} of filtration is induced by a set of formulas, that is  $\sim\; =\; \sim_\Delta$ for some   $\Delta$,
then the filtration $\ff{M}$ is called \emph{definable}.
Definable filtrations give more general transfer results than filtrations by means of arbitrary equivalence.
In particular, logics that admit definable filtration can be used to construct decidable extensions of Propositional Dynamic Logic
\cite{KikotShapZolAiml2020}, \cite{ExtCPDL-RogShap2022}.
In Theorems \ref{thm:wK4AF} and \ref{thm:ADF-Lm}, the equivalence relations were defined semantically.

\smallskip

\begin{problem}
For $m>0$, do logics $\Lm$ admit definable filtration?
\end{problem}

\paragraph{Filtrations via local tabularity and closure conditions}

Local tabularity of $\Lm$-clusters was crucial for building filtrations on $\Lm$-frames.
Another component was existence of the corresponding closure.

We say that a class $\clF$ is {\em closable}, if
for any frame $(X,R)$ there exists the smallest relation $R^{\clF}$ containing $R$ such that $(X,R^{\clF})\in \clF$.

Many examples of such classes are given by universal Horn sentences.

\smallskip

\begin{problem}
Let $\clF$ be a modally definable class of frames,\IShLater{Define in more exact form}
$\clC$ the class of clusters occurring in frames in $\clF$. Assume that $\Log(\clC)$ is locally tabular and $\clF$ is closable.
Does $\clF$ admit filtration?
\end{problem}

\paragraph{Subframe $m$-transitive logics}
There are weaker  than $\Lm$  subframe pretransitive logics. For example, consider the property
$$
\AA x_0 \, x_1\,  x_2\, x_3 (x_0 R x_1 R x_2 R x_3 \to x_0=x_3\vee x_0 R x_2 \vee x_1 R x_3 \vee x_0 R x_3).
$$
Then the class of such frames is 2-transitive and is closed under taking substructures. The finite model property of the logic of such frames is  unknown.  Very recently, the finite model property was announced for a family of pretransitive subframe logics that are weaker than $\Lm$ \cite{Dvorkin}. 

\IShLater{In fact, all pretransitive subframe have a similar property; can be written down as a modal formula; I do not think we have time for this now, but perhaps in the next iteration...}

\smallskip

\begin{problem}
Let $\clF$  be a class of $m$-transitive frames closed under taking subframes. Does the logic of $\clF$ have the finite model property?
\end{problem} 


\paragraph{Complexity}
It is immediate that all logics $\Lm$ are $\PSpace$-hard, since they have $\LS{4}$ as a fragment, and $\LS{4}$ is $\PSpace$-hard \cite{Ladner77}.
For the logic $\wK4$ ($m=1$), the $\PSpace$ upper bound was established in \cite{CondSatOnline}.

\smallskip

\begin{problem}
What is the complexity of the logics $\Lm$ for $m>1$?
\end{problem}
We conjecture that the construction given in Theorem \ref{thm:selective} leads to $\PSpace$ upper bound for all $m$.

\IShLater{

\paragraph*{Selective filtrations: transfer results}

Selective filtration is perhaps the simplest argument for the finite model property of $\wK4$.
Also, it gives a counter model of height bounded by a polynomial in the length of the formula.
However, unlike epifiltrations, it does not imply the finite model property for the derived logics.

  Problem: Transfer results.
}

\bibliographystyle{amsalpha}
\bibliography{wk4}


\end{document}